\documentclass[12pt,a4paper]{amsart}

\usepackage[all]{xy}

\usepackage[ascii]{inputenc} % for compatibility, please do not use utf8, let alone latin1
\usepackage{amssymb}
\usepackage{mathrsfs}

\usepackage{xcolor}

\usepackage[shortlabels]{enumitem}
\setlist[enumerate,1]{label={(\Alph*)}}
\setlist[enumerate,2]{label={(\alph*)}}
\setlist[enumerate,3]{label={$\bullet_{\arabic*}$}}

\newenvironment{PROOF}[2][\proofname.]
   {\begin{proof}[#1]\renewcommand{\qedsymbol}{\textsquare$_{\mathrm{#2}}$}}
   {\end{proof}}

\newtheorem{theorem}{Theorem}[section]

\newtheorem{claim}[theorem]{Claim}

\newtheorem{corollary}[theorem]{Corollary}

\newtheorem{lemma}[theorem]{Lemma}
\newtheorem{observation}[theorem]{Observation}
\newtheorem{proposition}[theorem]{Proposition}

\theoremstyle{definition}

\newtheorem{convention}[theorem]{Convention}

\newtheorem{definition}[theorem]{Definition}
\newtheorem{discussion}[theorem]{Discussion}
\newtheorem{example}[theorem]{Example}

\newtheorem{fact}[theorem]{Fact}
\newtheorem{hypothesis}[theorem]{Hypothesis}
\newtheorem{problem}[theorem]{Problem}

\theoremstyle{remark}

\newtheorem{notation}[theorem]{Notation}
\newtheorem{question}[theorem]{Question}
\newtheorem{remark}[theorem]{Remark}

\newcommand{\cf}{\mathrm{cf}}

\newcommand{\arity}{\mathrm{arity}}

\newcommand{\sub}{\mathrm{sub}}

\newcommand{\Ht}{\mathrm{ht}}

\newcommand{\Ord}{\mathrm{Ord}}
\newcommand{\otp}{\mathrm{otp}}

\newcommand{\Rang}{\mathrm{Rang}}
\newcommand{\rang}{\mathrm{rang}}

\newcommand{\set}{\mathrm{set}}

\newcommand{\Sub}{\mathrm{Sub}}\newcommand{\subs}{\mathrm{subset}}
\newcommand{\suc}{\mathrm{suc}}

\newcommand{\affine}{\mathrm{Affine}} %Afine or Affine? Doubt Andres
\newcommand{\SPSF}{\mathrm{SPSF}}

\newcommand{\AP}{\mathrm{AP}}
\newcommand{\AB}{\mathrm{AB}}
\newcommand{\overa}{\overline{a}}
\newcommand{\overb}{\overline{b}}
\newcommand{\overc}{\overline{c}}

\DeclareMathOperator{\len}{lg}
\DeclareMathOperator{\Set}{Set}
\DeclareMathOperator{\inter}{inter}
\DeclareMathOperator{\TFAB}{TFAB}

\newcommand{\bff}{\mathbf{f}}

\newcommand{\bfI}{\mathbf{I}}

\newcommand{\bfV}{\mathbf{V}}

\newcommand{\bbL}{\mathbb{L}}
\newcommand{\bbN}{\mathbb{N}}
\newcommand{\bbP}{\mathbb{P}}
\newcommand{\bbQ}{\mathbb{Q}}

\newcommand{\bbZ}{\mathbb{Z}}

\newcommand{\but}{\underline{but}}
\newcommand{\Iff}{\underline{iff}}

\newcommand{\then}{\underline{then}}

\newcommand{\when}{\underline{when}}

\newcommand{\pred}{pred}

\newcommand{\sn}{\smallskip\noindent}

\newcommand{\cC}{\mathscr{C}}

\newcommand{\cL}{\mathscr{L}}

\newcommand{\cT}{\mathscr{T}}

\newcommand{\varp}{\varepsilon}

\newcommand{\rest}{\restriction}

\newcount\skewfactor
\def\mathunderaccent#1#2 {\let\theaccent#1\skewfactor#2
\mathpalette\putaccentunder}
\def\putaccentunder#1#2{\oalign{$#1#2$\crcr\hidewidth
\vbox to.2ex{\hbox{$#1\skew\skewfactor\theaccent{}$}\vss}\hidewidth}}

\newbox\noforkbox \newdimen\forklinewidth
\setbox0\hbox{$\textstyle\bigcup$}
\setbox1\hbox to \wd0{\hfil\vrule width \forklinewidth depth \dp0
                        height \ht0 \hfil}
\setbox\noforkbox\hbox{\box1\box0\relax}
\def\unionstick{\mathop{\copy\noforkbox}\limits}
\def\nonfork#1#2_#3{#1\unionstick_{\textstyle #3}#2}
\def\nonforkin#1#2_#3^#4{#1\unionstick_{\textstyle #3}^{\textstyle
    #4}#2}
\setbox0\hbox{$\textstyle\bigcup$}
\setbox1\hbox to \wd0{\hfil{\sl /\/}\hfil}
\setbox2\hbox to \wd0{\hfil\vrule height \ht0 depth \dp0 width
                                \forklinewidth\hfil}
\newbox\doesforkbox
\setbox\doesforkbox\hbox{\box1\box0\relax}
\def\nunionstick{\mathop{\copy\doesforkbox}\limits}

\def\fork#1#2_#3{#1\nunionstick_{\textstyle #3}#2}
\def\forkin#1#2_#3^#4{#1\nunionstick_{\textstyle #3}^{\textstyle
    #4}#2}

\newcommand{\stickT}{%
\setbox255=\hbox{\raise1ex\hbox{$\hspace{0.2pt}\,\bullet\,$}}
\mathord{\rlap{\hbox to\wd255{\hss\hbox{$|$}\hss}}
\box255}
}
\newcommand{\stickS}{%
\setbox255=\hbox{\raise0.6ex\hbox{$\scriptstyle\bullet$}}
\mathord{\rlap{\hbox to\wd255{\hss\hbox{$\scriptstyle|$}\hss}}
\box255}
}

 \usepackage{hyperref}

\author[M. Asgharzadeh]{Mohsen Asgharzadeh}
\address{Hakimiyeh, Tehran, Iran.}
\email{mohsenasgharzadeh@gmail.com}

\author[M. Golshani]{Mohammad Golshani}
\address{School of Mathematics\\
 Institute for Research in Fundamental Sciences (IPM)\\
 P.O.\ Box:
19395--5746, Tehran, Iran.}
\email{golshani.m@gmail.com}

\author[S. Shelah]{Saharon Shelah}
\address{Einstein Institute of Mathematics,
The Hebrew University of Jerusalem,
9190401, Jerusalem, Israel; and\\
Department of Mathematics,
Rutgers University,
Piscataway, NJ 08854-8019, USA}
\urladdr{https://shelah.logic.at/}
\email{shelah@math.huji.ac.il}
\thanks{This work is based upon research funded by Iran National Science Foundation (INSF) under project No. 4027168. The third author thanks an individual who wishes to remain anonymous for generously funding typing services. Research partially supported by the Israel Science Foundation (ISF) grant no: 1838/19, and Israel Science Foundation (ISF) grant no: 2320/23;
	Research partially supported by the grant ``Independent Theories'' NSF-BSF, (BSF 3013005232).
The reader should note that the version in the author's  website is usually more up-to-date than the one in arXiv.
This is publication number
1246
in Saharon Shelah's list.
}

\makeatletter
\@namedef{subjclassname@2020}{\textup{2020} Mathematics Subject Classification}
\makeatother
\date{9-05-2022} % add date of last edit here, do not use \today

\title[Absolute co-Hopfianity]{Expressive Power of Infinitary Logic and Absolute co-Hopfianity}

\newcount\skewfactor
\def\mathunderaccent#1#2 {\let\theaccent#1\skewfactor#2
	\mathpalette\putaccentunder}
\def\putaccentunder#1#2{\oalign{$#1#2$\crcr\hidewidth
		\vbox to.2ex{\hbox{$#1\skew\skewfactor\theaccent{}$}\vss}\hidewidth}}

\date{\today}
\begin{document}

\begin{abstract}
Recently, Paolini and Shelah have
constructed absolutely Hopfian
torsion-free abelian groups of any given size. In contrast, we show that this is not necessarily the case for absolutely co-Hopfian groups.  We use the  infinitary logic to show that there are no absolute co-Hopfian abelian groups above the first beautiful cardinal. An extension of this result to the category of modules over a commutative ring is given.
\end{abstract}
\maketitle
\numberwithin{equation}{section}\tableofcontents

\section{Introduction}

	An abelian group	$G$ is called \emph{Hopfian} (resp. \emph{co-Hopfian}) if its surjective (resp. injective) endomorphisms are automorphisms. In other words, the co-Hopfian property of groups is, in some sense,  dual to the Hopfian groups.
 These groups were first considered by Baer in \cite{baer}, under different names. Hopf \cite{hopf} himself  showed that the fundamental group of closed two-dimensional orientable surfaces are Hopfian.

 There are a lot of interesting research papers in this area. Here, we recall only a short list of them. Following the book \cite{comb}, Hopf in 1932, raised the question as to whether a finitely generated group can be isomorphic to a proper factor of itself. For this and more observations,  see \cite{comb}.
 Beaumont \cite{B}
proved  that if $G$ is an abelian group of finite rank
 all of whose elements have finite order, then $G$ has no proper isomorphic
 subgroups. Kaplansky  \cite{KAP} extended this to modules over  a commutative principal ideal ring $R$ such that every proper residue class ring of $R$ is finite. Beaumont and Pierce \cite{Be} proved  that if $G$ is co-Hopfian, then the torsion part of $G$  is of size at most continuum, and further that $G$ cannot be a $p$-groups of size $\aleph_0$. This naturally left open the problem of the existence of co-Hopfian $p$-groups of uncountable size $\leq 2^{\aleph_0}$, which was later solved by Crawley \cite{crawley} who proved that there exist  $p$-groups of size $2^{\aleph_0}$.
 One may attempt to construct (co-)Hopfian
 groups of large size by taking a huge direct sum of (co-)Hopfian groups.
 In this regard,  Baumslag \cite{ba}  asked when  the direct sum
 of two (co-)Hopfian groups is again (co)-Hopfian.  Corner \cite{cor} constructed two torsion-free abelian
 Hopfian groups which have non-Hopfian direct sum. See  \cite{GF}, for more on this and its connections with the study of algebraic entropy.

Despite its long history, only very recently
the problem of the existence of uncountable (co-)Hopfian
abelian groups was solved, see \cite{AGSa} and  \cite{1214}. For instance, in view of \cite[Corollary 4.13]{AGSa} and
for any cardinals $\lambda> 2^{\aleph_{0}}$, we observe that there is a co-Hopfian abelian group $G$ of size $\lambda$ iff $\lambda= \lambda^{\aleph_{0}}$.
The usual construction of (co)-Hopfian groups is not absolute, in the sense that they may lose their property in some generic extension of the universe, and the problem of giving more explicit and constructive
constructions of such groups has raised some attention in the literature. For example, while it is known from the work of Shelah \cite{Sh:44} that there are indecomposable abelian groups of any infinite size, the problem of the existence of arbitrary large absolutely indecomposable groups is still open, see  \cite{Nad94}. It is worth noticing  that indecomposability implies the Hopfian
property. Another interpretation of  more explicit construction  is  provable without the axiom of choice.

Recall that a group	$G$  \emph{absolutely co-Hopfian} if it is co-Hopfian in any
further	generic extension of the universe.
Similarly, one may define  \emph{absolutely  Hopfian} groups.
As far as we know several researchers have considered the following problem:

\begin{problem}\label{prob}\begin{itemize}
		\item[(i)] (see e.g. \cite[Page 535, Problem (3)]{1214})
	 Is it possible to construct absolutely  Hopfian
	torsion-free groups of a given size?
	
	\item[(ii)]	(Shelah) Is it possible to construct absolutely co-Hopfian
	torsion-free groups of a given size? \end{itemize}
\end{problem}

Recently,
Paolini and Shelah  \cite[Theorem 1.3]{1214}
constructed absolutely Hopfian
torsion-free groups of any given size $\lambda$, thereby confirming Problem \ref{prob}(i) in positive direction. It seems in some sense that Hopfian and co-Hopfian groups are dual to each ether, one may predict that there is a connection between Problem \ref{prob} (i) and (ii). But,  any such dual functor, may enlarge or collapse the cardinality of the corresponding groups, hence we can not use the ideas behind the duality to answer Problem \ref{prob}(ii). For example, Braun and Str\"ungmann \cite{independence_paper} showed that
the existence of  infinite abelian  $p$-groups of
size $\aleph_0 < |G| < 2^{\aleph_0}$ of the following types are independent of ZFC:
\begin{itemize}
	\item[(a)] both Hopfian and co-Hopfian,
	\item[(b)] Hopfian but not co-Hopfian,
	\item[(c)] co-Hopfian but not Hopfian.
\end{itemize}Also, they proved that the above three
types of groups of size $2^{\aleph_0}$ exist in ZFC.
In \cite{Sh:110}, Shelah studied and  coined the concept of a \emph{beautiful cardinal},  denoted   by $\kappa_{\rm{beau}}$, which is a kind of Ramsey cardinal (see Definition \ref{beau}). This cardinal has an essential role in the study of  absolutely  endorigid groups.
Indeed, according to \cite{Sh:880}, for any infinite cardinal $\lambda < \kappa_{\rm{beau}}$,    there is a torsion-free absolutely endorigid abelian group $G$ of size $\lambda$. By definition,  for any morphism $f:G\to G$ there is an integer $n$
so that $f(g)= ng$ for all $g\in G$. In particular, if
$G$ is co-Hopfian, then $G$ is divisible.
For the remaining cardinals $\lambda \geq \kappa_{\rm{beau}}$ the existence of an  absolutely  endorigid group of size $\lambda$ is one of the most challenging problems in
the theory infinite abelian groups.

The
(co-)Hopfian property easily can be extended to the context of modules over commutative rings and even to the context of sheaves over schemes. However, compared to the case of abelian groups, and up to our knowledge, there are very few results for modules. For example, a funny result of Vasconcelos \cite[1.2]{vo} says that any surjective $f:M\to M$ is an isomorphism, where $M$ is  a noetherian  module over a commutative and noetherian ring $R$.
As a geometric sample, let $X$ be an algebraic variety over an algebraically closed
field. If a morphism $f : X\to X$ is injective then  a result of   Ax and   Grothendieck indicates that  $f$ is bijective, see Serre's exposition \cite{ser}.

In contrast  to the case of Hopfian groups (see  \cite[Theorem 1.3]{1214}), we use  the additive frame to show that there are no
absolutely co-Hopfian groups above the beautiful cardinal, providing
a partial solution to Problem \ref{prob}(ii). Indeed, we prove a more general
result in the context of additive $\tau$-models (see Definition \ref{a2}), which includes in particular
the cases of abelian groups and $R$-modules, for an arbitrary commutative ring $R$:

\begin{theorem}\label{1.2}
The following assertions are valid:
\begin{enumerate}  \item    If $M$ is an abelian group of cardinality  greater or equal $ \kappa := \kappa_{\rm{beau}}$, then $M$ is not absolutely co-Hopfian.
	
	\item	If $M$ is an $R$-module of cardinality  greater or equal $ \kappa=\kappa_{\rm{beau}}(R, \aleph_0)$,
	then $M$ is not absolutely co-Hopfian.
	
		\item	If $M$ is an additive $\tau$-model  of cardinality  greater or equal $ \kappa=\kappa_{\rm{beau}}(\tau)$,
		then $M$ is not absolutely co-Hopfian.
\end{enumerate}
\end{theorem}

The organization of this paper is as follows.
 Szmielew \cite{sz} developed first order theory of abelian groups, see also \cite{Sh:54} for further discussions. In
Section 2  we review
infinitary languages and develop some parts of  abelian group theory in this context.  For this, we
use the concept of $\theta$-models from \cite{Sh:977}. We also introduce some sublanguages of the infinitary languages, which play some role in our later investigation.

For
Section 3, let us
   fix a pair $(\lambda, \theta)$ of regular cardinals   and let  $\kappa$ be as Theorem \ref{1.2}. The new object studied here is called the general frame and its enrichment the  \emph{additive frame}. Such a frame is of the form
	$$\bff:=(M, \cL, \lambda, \kappa, \theta, \Omega),$$
where $M$ comes from Theorem \ref{1.2}, and it has an additive  $\tau_{M}$-model of  cardinality $\geq \kappa$ and $\cL$ is a class of  certain  formulas in the vocabulary $\tau_{M}$. For more details, see
Definition \ref{k2}.
 The main result of Section 3 is Theorem  \ref{L7}. This gives us  an additive  frame $\bff$.

 Section 4 is about the concept of \emph{algebraic closure} in a frame (see Definition \ref{n8}).
 This enables us to improve Theorem  \ref{L7},
 which is needed in the sequel. For instance, see Lemma \ref{n11}.
%We investigate $(< \kappa)$-algebraically closed sequences and make some tries to build  endomorphisms.

%Such formulas appear in universal algebra.
%We advance seemingly finding new cases but not clear if enough to eliminate quantifiers.

 In Section 5 we put all things together and present the proof of Theorem \ref{1.2}.
Let $\chi:=|M|\geq\kappa$, and let $\mathbb{P}:=\rm{Col}(\aleph_{0},\chi)$.
Forcing with $\mathbb{P}$ enables us to collapse
$|M|$ into  $\aleph_0$, i.e., for any $\mathbb{P}$-generic filter $G_{\mathbb{P}}$
over $V$, we have
\begin{center}
	$V[G_{\mathbb{P}}]\models$``$M$ is  countable''.
\end{center}
We  show  in $V[G_{\mathbb{P}}]$, there exists a 1-1 map $\pi:M\to M$ which is not surjective.

We close  the introduction by noting that all groups (resp. rings) are abelian (resp. commutative), otherwise specialized, and
our notation is standard and follows that in Fuchs books \cite{Fuc73} and \cite{fuchs} and Eklof-Mekler \cite{EM02}.

\section{Infinitary languages}\label{2}

In the first subsection, we briefly review infinitary languages  and the concept of additive $\theta$-models, introduced in \cite{Sh:977}.
In the second subsection, we present basic properties  of affine subsets, i.e., ones closed under $x - y +z$.

\subsection{A review  of infinitary languages}
In this subsection we briefly review the infinitary logic, and refer to \cite{Bar73} and \cite{Sh:977}
for more information.

\iffalse
We may wonder  which set  $\cL$ of formulas is relevant for $M \in \AB$ or just additive structures?

In \ref{a11} on we treat general $\cL \subseteq \bbL_{\infty, \infty}$ not just existential positive conjunctive as in \cite{Sh:977}. Earlier (\ref{a2}) we define additive $\theta$-structures, and phrase the relevant statements (in \ref{a5}) and discuss various things including relevant second order quantifiers (see \ref{a9}) and in \ref{a11} we formalize the second order version.
\fi
\begin{convention}
Given a model $M$, by $\tau_M$ we mean the language (or vocabulary) of the model $M$.
\end{convention}

\begin{notation}
\begin{enumerate}
\item By $\AB$ we mean the class of abelian groups.

\item Given a vocabulary $\tau$ which contains two place functions $+, -$, we define the affine operation $\affine( {x},   {y} ,  {z})$ as the three place
function $\affine( {x},   {y} ,  {z}):= {x} -  {y} +  {z}$.
\end{enumerate}
\end{notation}

\begin{definition}\label{a2}
	Let $M$ be a model of  vocabulary $\tau_M$.
\begin{enumerate}\item We say  $M$ is an \emph{additive $\theta$-model} when:
        \begin{enumerate}[(a)]
            \item the two place function symbols $+, -$ and the constant symbol $0$ belong to $\tau_{M},$

            \item $G_{M} = (\vert M \vert, +^{M}, -^{M}, 0^{M})\in \AB$,

            \item $R^{M}$ is a subgroup of ${}^{n}(G_{M})$, for any predicate symbol $R \in \tau_M$   with $\rm{arity(R)} = n$,

            \item $F^{M}$ is a homomorphism from ${}^{n}(G_{M})$ into $G_{M}$, for any function symbol $F \in \tau_M$   with arity $n,$

            \item $\tau_{M}$ has cardinality $\leq \theta.$ %we may omit $\theta$ if $\theta$ is clear from the context.
        \end{enumerate}
\item For  an  additive $\theta$-model  $M$, we say   $X \subseteq {M}$ is \it{affine} if
    $X$ is closed  under the affine  operation $\affine( {x},   {y} ,  {z})$. In other words, ${a} - {b} + {c} \in X$
    provided that
    ${a}, {b},{c} \in X$.

\item We say  $M$ is an  affine $\theta$-model provided:
     \begin{enumerate}[(a)]
    	\item we do not necessarily have  $+, -,0$ in the vocabulary, but only the three place
    	function $\affine(x,y,z)$,
    	\item
    	if $R\in\tau_M$ is  an  $n$-place predicate and $\overa_{l}= \langle a_{l, i}: i<n  \rangle\in R^M$ for $l=0,1,2$ and $$\overb:=\affine(\overa_0, \overa_1, \overa_2)=\big\langle \affine(a_{0, i},a_{1, i},a_{2, i}): i<n\big\rangle,$$ then $\overb\in R^M$,
    	
    	\item for any $n$-place function symbol $F \in \tau_M$ and $\overa_{l}= \langle a_{l, i}: i<n  \rangle\in {}^{n}M$, for $l=0,1,2$,
    we have
    \[
    F^M(\affine(\overa_0, \overa_1, \overa_2)) = \affine(F^M(\overa_0), F^M(\overa_1), F^M(\overa_2)),
    \]

    	\item $\tau_{M}$ has cardinality $\leq \theta.$
    \end{enumerate}
 \item
Suppose $M$ is an  affine $\theta$-model. We say
$M$ is truly affine
provided
for some fixed $a\in M$ and for the following interpretation
\begin{itemize}
	
	\item  $x+y :=\affine(x,a,y)= x-a+y,$
	
	\item $x-y := \affine(x,y,a)= x-y+a,$
	
	\item $0:=a$,
\end{itemize}
 then	we get an abelian group, and hence an additive $\theta$-model.
 \end{enumerate}
 We may omit $\theta$ if it is clear from the context.
 \end{definition}

\begin{remark}
i) A natural question arises: Is  any affine $\theta$-model  truly affine? Not necessarily this holds, see Example \ref{ntrueaf}, below.

ii) More generally, we can replace  $\{+,-,\affine\}$ for a set $\tau_f$ of beautiful function from \cite{61}.
The corresponding  result  holds in this frame.
	\end{remark}

\begin{example}\label{ntrueaf}
Let $G$ be an abelian group, $H$ be a proper subgroup of it and $a \in G\setminus H.$ Define $M$ as follows:
\begin{itemize}
	
	\item the universe of $ M$ is $a+H,$
	
	\item $\tau_{M} := \{+, -, \affine\}$,

\item $+^M$ and $-^M$ are $+^G \restriction M$ and $-^G \restriction M$ respectively,
	
	\item $\affine^{M}:= \affine^G \restriction M$, where $\affine^G=\{x-y+z: x, y, z \in G   \}$.
\end{itemize}
Then the following two assertions hold:
\begin{itemize}
	
	\item[a)] $M$ is an  affine $\aleph_0$-model, isomorphic to $H$.
	
	\item[b)] $M$ is not  an
abelian group.
\end{itemize}
	\end{example}
 \begin{definition}
 \label{a22}  (1) We say a class $K$ of models is an \emph{additive $\theta$-class},  when
 	 $M$ is an additive $\theta$-model for all
 	 $M  \in K$, and $$ \tau_{M} = \tau_{N}\quad \forall M,N  \in K.$$ We denote the resulting common language  by $\tau_K.$

(2) Similarly, one can define  affine $\theta$-classes.
\end{definition}

\begin{hypothesis}\label{m2}
Let	$\Omega$ be a set of cardinals with $1\in\Omega$ and members of $\Omega\setminus \{1\}$ are  infinite cardinals.
\end{hypothesis}

\begin{notation}\label{z2}
	\begin{enumerate}
		\item By $\overline{x}_{[u]}$ or $\overline{x}_{u}$  we mean $\langle x_{\alpha}: \alpha \in u \rangle.$ So, with no repetition.
		
		\item Suppose $\varphi(\overline{x}_{u})$ is a formula. By $\varphi(M)$ we mean $ \{ \overa \in {}^{u}M: M \models \varphi[\overa] \}.$ \item For a formula $\varphi(\overline{x}_{[u]}, \overline{y}_{[v]})$ and $\overb \in {}^{v}M,$ we let $$\varphi(M, \overb) := \big\{ \overa \in {}^{u}M: M \models \varphi[\overa, \overb] \big\}.$$
		
		%\item $\AB$ is the class of Abelian groups.
		
		\item Given a sequence $t$, by $\rm{lg}(t)$ we mean the length  of $t$. \end{enumerate}
\end{notation}

\begin{definition}
	Suppose $\kappa$ and $\mu$
	are infinite cardinals, which we allow to be $\infty$. The infinitary language $\mathcal{L}_{\mu, \kappa}(\tau)$
	is defined so as its vocabulary is the same as $\tau,$ it has the same terms and atomic formulas as in $\tau,$ but we also allow conjunction and disjunction of length less than $\mu$, i.e., if $\phi_j,$ for $j<\beta < \mu$ are formulas, then so are $\bigvee_{j<\beta}\phi_j$ and $\bigwedge_{j<\beta}\phi_j$. Also, quantification over less than $\kappa$ many variables (i.e., if $\phi=\phi((v_i)_{i<\alpha})$, where $\alpha < \kappa$, is a formula, then so are $\forall_{i<\alpha}v_i \phi$ and $\exists_{i<\alpha}v_i\phi$).
\end{definition}
Note that $\mathcal{L}_{\omega, \omega}(\tau)$ is just the first order logic with vocabulary $\tau.$
Given $\kappa$, $\mu$ and $\tau$ as above, we are sometimes interested in some special formulas from $\mathcal{L}_{\mu, \kappa}(\tau)$.
\begin{definition}
\label{a21}
	Suppose $\kappa$ and $\lambda$ are infinite cardinals or possibly $\infty$. We define the logic $\cL_{\lambda, \kappa,\Omega}$ as follows:
	\begin{enumerate}
		\item For a vocabulary $\tau$, the language
$\cL_{\lambda, \kappa,\Omega}(\tau)$ is defined as the set of formulas with $<\kappa$ free variables (without loss of generality they are subsets of $\{x_{\zeta}:\zeta<\kappa\}$, see Discussion \ref{dis1})
	which is the closure of the set of basic formulas , i.e., atomic and the negation of atomic formulas, under:
	
	\begin{itemize}
	
	\item[(a)] conjunction of $<\lambda $
	formulas,% (together with $<\theta$ free variables)
	\item [(b)] disjunction of $<\lambda $
	formulas, % (together with $<\theta$ free variables)
	
	\item [(c)] For any $\sigma\in\Omega$, 	\begin{itemize}
		\item[($c_{1}$)] $\varphi(\overline{x} ):=(\exists^{\sigma}\overline{x}')\psi(\overline{x},\overline{x}')$, or
		\item[($c_{2}$)] $\varphi(\overline{x} ):=(\forall^{\sigma}\overline{x}')\psi(\overline{x},\overline{x}')$,
\end{itemize}
 where $\psi(\overline{x},\overline{x}')$ is a formula.
We usually omit $\sigma$,  if  $\sigma=1$.
\end{itemize}
We usually omit $\Omega$ if it is clear from the context.
		\item
Satisfaction is defined as usual, where for the formulas $\varphi(\overline{x} ):=(\exists^{\sigma}\overline{x}')\psi(\overline{x},\overline{x}')$ and
		 $\varphi(\overline{x} ):=(\forall^{\sigma}\overline{x}')\psi(\overline{x},\overline{x}')$, it is defined as:
\begin{itemize}
	
	\item[(a)] If $\varphi(\overline{x}):=(\exists^{\sigma}\overline{x}')\psi(\overline{x},\overline{x}')$,
	 $M$ is a $\tau$-model, and $\overa \in {}^{\lg(\overline{x})}M$, then  $M\models\varphi[\overline{a}]$ if and only if there are
	$\overb_\varp \in {}^{\lg(\overline{x}')}M$ for all $\varp<\sigma$  pairwise distinct  such that $M\models \psi[\overline{a},\overline{b}_{\varp}]$ for all $\varp<\sigma$.

		\item[(b)] If $\varphi(\overline{x} ):=(\forall^{\sigma}\overline{x}')\psi(\overline{x},\overline{x}')$, then $$M\models\varphi(\overline{x} )  \iff M\models \neg\big[\exists^{\sigma}\overline{x}' \neg\big(\psi(\overline{x},\overline{x}')\big)\big].$$
Note that $\neg(\psi(\overline{x},\overline{x}'))$ is not necessarily in $\cL_{\lambda, \kappa,\Omega}(\tau)$.
\end{itemize}
\end{enumerate}
\end{definition}

\begin{remark}It may be worth to mention that
$\cL_{\lambda, \kappa,\Omega}(\tau)$
is a generalization of the infinitary language $\mathcal{L}_{\mu, \kappa}(Q)$ when
$\Omega:=\{1,\aleph_1\}$.
\end{remark}
\begin{discussion}
\label{dis1}
Given a formula $\varphi$ in $\mathcal{L}_{\infty, \theta}(\tau)$ with free variables $\overline{x}_{\varphi},$ we can always assume that  $\overline{x}_{\varphi}=\langle x_\zeta:\zeta\in u_\varphi\rangle, $ for some $u_\varphi\in[\theta]^{<\theta}$. The key point is that if $\varphi=\varphi(\overline{x}),$ where $\overline{x}= \langle x_\zeta:\zeta\in w \rangle,$ where $w$ is a set of ordinals of size less than $\theta$, and if $f: w \leftrightarrow u$
is a bijection where $u \in [\theta]^{<\theta}$, and
$
\psi(\overline{x}) \equiv \Sub_f^{\overline{x}}(\varphi),
$
where $\Sub_f^{\overline{x}}(\varphi)$ is essentially the formula obtained from $\varphi$ by replacing the variable $x_\zeta$ by $x_{f(\zeta)}$, then if $\bar a \in$$^{w}M$, $\bar b \in$$^{u}M$
and $a_\zeta = b_{f(\zeta)}$, for $\zeta \in w$, then
\[
M \models \varphi[\bar a] \iff M \models \psi[\bar b].
\]
We can similarly assume that all bounded variables are from $\{x_i: i<\theta  \}$.
\end{discussion}

	\begin{convention}
In what follows, saying closed under $\exists$  (resp. $\forall$) means under all $\exists^{\sigma}$ (resp. $\forall^{\sigma}$).
\end{convention}
In the next definition, we consider some classes of infinitary formulas that we will work with them latter.
\begin{definition}
	\label{def1}
	Suppose $\theta$ is an infinite cardinal, or $\infty,$ and suppose $\tau$ is a language. Here, we collect some
	infinitary subclasses of the language $\cL_{\infty, \theta}(\tau)$:
	\begin{enumerate}
		\item $\cL_{\infty, \theta}^{\rm{cop}}(\tau)$ is the class of   conjunction-positive formulas, i.e., the closure of atomic
		formulas under $\bigwedge, \exists, \forall$.
		
		\item  $\cL_{\infty, \theta}^{\rm{cpe}}(\tau)$  is the class of   conjunction-positive existential formulas, i.e., the closure of atomic
		formulas under $\bigwedge$ and $\exists$.
		
			\item  $\cL_{\infty, \theta}^{\rm{co}}(\tau)$  is the closure of atomic formulas and $x_i \neq x_j$ under  $\bigwedge$, $\exists$ and $\forall$.
		
			\item  $\cL_{\infty, \theta}^{\rm{ce}}(\tau)$  is the closure of atomic formulas and $x_i \neq x_j$ under $\bigwedge$ and $\exists$.
	\end{enumerate}
\end{definition}
We shall use freely the following simple fact.
\begin{fact}
\label{lem1}
$\cL_{\infty, \theta}^{\rm{co}}(\tau) \supseteq \cL_{\infty, \theta}^{\rm{cop}}(\tau) \cup \cL_{\infty, \theta}^{\rm{ce}}(\tau) \supseteq \cL_{\infty, \theta}^{\rm{cop}}(\tau) \cap \cL_{\infty, \theta}^{\rm{ce}}(\tau) \supseteq \cL_{\infty, \theta}^{\rm{cpe}}(\tau)$.
\end{fact}

The following lemma is easy to prove.

\begin{lemma}\label{a5}
	Assume $M$ is an additive $\theta$-model, $\tau = \tau_{M}$ and
	$\varphi(\overline{x}_u) \in \cL_{\infty, \infty}(\tau)$ with  $\varp=\rm{lg}(\overline{x})$. The following assertions are valid:
	\begin{enumerate}
		\item If $\varphi(\overline{x}_u) \in \cL_{\infty, \infty}^{\rm{cop}}(\tau)$, then
		$\varphi(M) $ is a subgroup of ${}^{u}M.$

		\item If $\varphi(\overline{x}_u) \in \cL^{\rm{cpe}}_{\infty, \infty}(\tau)$,
		$f \in \rm{End}(M)$ and $M \models \varphi[\overline{a}]$, then $M \models \varphi[f(\overline{a})].$
		
		\item If $\varphi(\overline{x}_u) \in \cL^{\rm{cpe}}_{\infty, \theta}(\tau)$, $M, N$ are $\tau$-models and $f: M \to N$ is a homomorphism, then $f$ maps $\varphi(M)$
		into $\varphi(N).$
		
		\item If $\varphi(\overline{x}_u) \in \cL^{\rm{ce}}_{\infty, \theta}(\tau)$, $M, N$ are $\tau$-models and $f: M \to N$ is a 1-1 homomorphism, then $f$ maps $\varphi(M)$
		into $\varphi(N).$
		\iffalse
		\item Assume $\varphi(\overline{x}_{[u]}) \in \mathfrak{L}_{\infty, \infty}(\tau)$ is  positive (= closure of atomic under $\{ \wedge, \vee, \exists, \forall \}$). \underline{Then}:
		
		\begin{enumerate}
			\item[$\boxplus$] If $G$ is an abelian group,  $f$ is an epimorphism and $G \models \varphi[\overline{a}_{[u]}]$ \then \ $G \models \varphi[f(\overline{a})].$
		\end{enumerate}
		\fi

        \item If $\varphi(\overline{x}_u) \in \cL^{\rm{co}}_{\infty, \theta}(\tau)$, $M, N$ are $\tau$-models and $f: M \to N$ is a bijection, then $f$ is an isomorphism from $\varphi(M)$
		onto $\varphi(N).$

\item Assume $\psi(\bar y)$ is obtained from $\varphi(\bar x)$ by adding dummy variables, permuting the variables and substitution not identifying variables. Then
\[
\psi(\bar y) \in \cL^{\ast}_{\infty, \theta}(\tau) \iff  \varphi(\bar x) \in \cL^{\ast}_{\infty, \theta}(\tau),
\]
where $\ast \in \{\rm{cop}, \rm{cpe}, \rm{ce}, \rm{co}   \}$.
	\end{enumerate}
\end{lemma}
\begin{proof}
The proof is  by induction on the complexity of the formulas. For completeness, we sketch the proof.
If the formula is an atomic formula, then it is evident that all of the above items are satisfied. It is also easy to see that each item is preserved under $\bigwedge$, in the sense that
if $\psi= \bigwedge_{i \in I}\varphi_i$ is well-defined and the lemma holds for each $\varphi_i$, then it holds for $\psi$.

We now consider the case where $\psi(\bar x)= (\exists^\sigma \bar y)\varphi(\bar x, \bar y)$,
and assume the induction hypothesis holds for $\varphi.$
We consider each clause separately, assuming in each case, the formula $\varphi$ is in the assumed language.
\begin{enumerate}
\item[]{\bf Clause (1)}: Suppose $\varphi(M)$ is a subgroup of $^{\rm{lg}(\bar x)+\rm{lg}(\bar y)}M$. We show that
 $\psi(M)$ is a subgroup of $^{\rm{lg}(\bar x)}M$. To see this, let $\bar a_0, \bar a_1 \in \psi(M)$. Then for some $\bar b_0$
 and $\bar b_1$ we have $$M \models `` \varphi[\bar a_0, \bar b_0]\emph{ and }\varphi[\bar a_1, \bar b_1]''.$$ By  induction,
 $M \models$$\varphi[\bar a_0-\bar a_1, \bar b_0-\bar b_1]$,
 hence $M \models \psi[\bar a_0-\bar a_1].$ Thus $\bar a_0-\bar a_1 \in \psi(M)$.

\item[]{\bf Clause (2)}: Suppose $M \models$$\psi[\bar a]$. Then for some $\bar b$, we have
$M \models \varphi[\bar a, \bar b]$. By the induction,  $M \models \varphi[f(\bar a), f(\bar b)]$,
and hence $M \models \psi[f(\bar a)]$, as requested.

\item[]{\bf Clause (3)}: As in clause (2), we can show that if $M \models$$\psi[\bar a]$,
then $N \models$$\psi[f(\bar a)]$, and this gives the required result.

\item[]{\bf Clause (4)}: As in clause (3). The assumption of $f$ being 1-1 is used to show that if $x_i \neq x_j$, then $f(x_i) \neq f(x_j)$.

\item[]{\bf Clause (5)}: As in clause (4).

\item[]{\bf Clause (6)}: This is easy.
\end{enumerate}

Finally, suppose that  $\psi(\bar x)= (\forall^\sigma \bar y)\varphi(\bar x, \bar y)$, and assume the induction hypothesis holds for $\varphi$. We only have to consider items (1) and (5).
\begin{enumerate}
\item[]{\bf Clause (1)}: Suppose $\varphi(M)$ is a subgroup of $^{\rm{lg}(\bar x)+\rm{lg}(\bar y)}M$. We show that
 $\psi(M)$ is a subgroup of $^{\rm{lg}(\bar x)}M$. To see this,
 let $\bar a_0, \bar a_1 \in \psi(M)$. We have to show that $\bar a_0$-$\bar a_1 \in \psi(M)$.   Thus let $\bar b \in$$^{\rm{lg}(\bar y)}M$.
By the induction hypothesis, $$M \models `` \varphi[\bar a_0, \bar b]\emph{ and }\varphi[\bar a_1, \bar 0]''.$$ Thanks to  induction,
 $M \models$$\varphi[\bar a_0-\bar a_1, \bar b-\bar 0]$. As this holds for all $\bar b$,
 we have  $M \models \psi[\bar a_0-\bar a_1].$  Thus $\bar a_0-\bar a_1 \in \psi(M)$, as requested.

\item[]{\bf Clause (5)}: As before, we can easily show that $f$ maps $\psi(M)$
		into $\psi(N).$ To see it is onto, let $\bar c \in \psi(N)$. Then $N \models \psi[\bar c]$. As $f$ is onto,
for some $\bar a$ we have $\bar c=f(\bar a)$. We have to show that $\bar a \in \psi(M)$. Thus let $\bar b \in$$^{\rm{lg}(\bar y)}M$.
Then $\bar d=f(\bar b) \in$$^{\rm{lg}(\bar y)}N$, and by our assumption, $N \models \varphi[\bar c, \bar d]$. As $f$ is an isomorphism, $M \models \varphi[\bar a, \bar b]$.
As $\bar b$ was arbitrary, $M \models \psi[\bar a]$, i.e., $\bar a \in \psi(M)$.

\end{enumerate}
The lemma follows.
\end{proof}

Let us restate the above result in the context of $R$-modules:
\begin{corollary}\label{L4}
Let  $M$ be an $R$-module.
\begin{enumerate}
	\item If $\varphi(\overline{x}_{u}) \in \cL_{\infty, \theta}^{\rm{{cpe}}},$
	then
	$\varphi(M) $ is an abelian subgroup of $ ({}^{u}M, +).$

\item Similar result holds for formulas $\varphi(\overline{x}_{u}) \in \cL_{\infty, \theta}^{\rm{{cop}}}.$
\end{enumerate}	
Furthermore, if $R$ is commutative, then in the above, $\varphi(M) $ becomes a submodule of $ ({}^{u}M, +).$
\end{corollary}
\begin{remark}
If $R$ is not commutative, then $\varphi(M) $ is not necessarily a submodule. To see this, suppose $a, b, c \in R$ are such that $abc \neq bac,$
and suppose $M$ is a left $R$-module. Define $\varphi(x,y):=``
y=ax''$.
Now note that $(c, ac) \in \varphi(M).$ If $\varphi(M)$ is a submodule, then we must have $(bc, bac) \in \varphi(M).$ Hence
$
bac = ax = abc,
$
which contradicts $abc \neq bac.$
\end{remark}

\subsection{More on affineness}\footnote{The results of this subsection are independent from the rest of
	the paper.}
In this subsection we try replacing subgroups by affine subsets.
The main result of  this subsection is Proposition
\ref{m8}. First,  we fix the hypothesis and present the corresponding definitions.
Here, affinity demand relates only to the formulas, not the content.
\begin{hypothesis}\label{m2}
	\begin{enumerate}
		\item $R$ is a ring,
		
		\item $M$ is an $R$-module,
		
		\item Let
		 $\lambda, \kappa$ be regular and $\lambda \geq \kappa \geq \theta\geq |\Omega|+|\tau_M|,$  where $\Omega$ is a set of cardinals such that $1 \in \Omega$
		and all other cardinals in it are infinite.
	\end{enumerate}
\end{hypothesis}

\begin{definition}\label{m5a}
	Let $\affine_{1}$ be the set of all formulas $\varphi(\overline{x}) \in \cL_{\infty, \theta}(\tau_{M})$ so that $\rm{lg}(\overline{x}) < \theta$ and
	$\varphi(M)$ is closed  under $\overline{x} - \overline{y} + \overline{z}$. In other words, $\overline{a} - \overline{b} + \overline{c} \in \varphi(M)$
	provided that
	$\overline{a}, \overline{b}, \overline{c} \in \varphi(M)$.
\end{definition}
We now define another class $\affine_{2}$ of formulas of $ \cL_{\infty, \theta}(\tau_{M})$,
and show that it is included in $\affine_{1}$. To this end, we first make the following definition.
\begin{definition}
	\label{m3} Suppose $\alpha_*$ is an ordinal. Let $\varphi(\overline{x}, \overline{y})$ and $\overline{\psi}(\overline{x}, \overline{y}) = \langle \psi_{\alpha}(\overline{x}, \overline{y}): \alpha < \alpha_{*} \rangle$
	be a sequence of formulas from $ \cL_{\infty, \theta}(\tau_{M})$. Let $\overline{b} \in {}^{\len(\overline{x})}{M}$ and $\overline{a} \in {}^{\len(\overline{y})}{M}$.
	%Let also $$u:=\{i: x_i \emph{ appears in } \overline{x}   \}.$$
	Then we set
	\begin{enumerate}
		\item  $\set_{\overline{\psi}}(\overb, \overa)$ stands for the following set $$\set_{\overline{\psi}}(\overb, \overa):=\big\{ \alpha \in  \alpha_{*}:   (\overb^{\frown}\overa\in\psi_{\alpha}(M)) \big\}.$$
		
		\item   By  $\Set_{\varphi, \overline{\psi}}(\overa)$ we mean
		$$\Set_{\varphi, \overline{\psi}}(\overa):= \big\{ u \subseteq \alpha_{*}: \text{for some} \ \overc \in \varphi(M, \overa) \ \text{we have} \ u = \set_{\overline{\psi}}(\overc, \overa) \big\}.$$

		\item By $\inter_{\varphi, \overline{\psi}}(\overline{a})$ we mean
		$$\bigg\{ (w_{0}, w_{1}): w_{0} \subseteq w_{1} \subseteq \alpha_{*} \ \text{and } \exists\ u_{0}, u_{1} \in \Set_{\varphi, \overline{\psi}}(\overa) \ \text{s.t.} \ w_{1} \subseteq u_{1} \ \text{and} \ u_{0} \cap w_{1} = w_{0}  \bigg\}.$$
		In particular, 	we have the following flowchart: $$
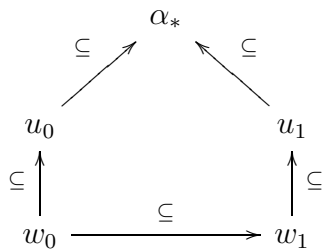
\xymatrix{
			&&\alpha_{*}\\	&u_{0}\ar[ur]^{\subseteq}&&u_{1}\ar[ul]_{\subseteq}\\&  w_{0}\ar[rr]^{ \subseteq}\ar[u]^{\subseteq}
			&&w_{1}\ar[u]_{\subseteq}
			&&&}$$
	\end{enumerate}
\end{definition}
We are now ready to define the class $\affine_{2}$ of formulas:

\begin{definition}\label{m5}
	Let $\affine_{2}$ be the closure of the set of atomic formulas by:
	
	\begin{enumerate}[(a)]
		\item arbitrary conjunctions,
		
		\item existential quantifier $\exists \overline{x},$ and
		
		\item suppose for a given ordinal $\alpha_*$, the formulas $\varphi(\overline{x}, \overline{y}), \langle \psi_{\alpha}(\overline{x}, \overline{y}): \alpha < \alpha_* \rangle$ are  taken from $\affine_{2}$  such that  $\varphi(\overline{x}, \overline{y}) \geq \psi_{\alpha}(\overline{x}, \overline{y})$  for all $\alpha < \alpha_*$. Also suppose that
		$$\Upsilon \subseteq \{ (w_{0}, w_{1}): w_{0} \subseteq w_{1} \subseteq \alpha_{*} \}.$$
		Then $\vartheta(\overline{y}) = \Theta_{\varphi, \overline{\psi}, \Upsilon}(\overline{y}) \in \affine_{2},$ where $\vartheta(\overline{y})$ is defined  such that
		$$M \models \vartheta[\overline{a}]\iff  \Upsilon \subseteq \inter_{\varphi, \overline{\psi}}(\overline{a}).$$
	\end{enumerate}
	\iffalse
	ii) Let $({\overline{\psi}},\alpha_{*})$
	be as above.  The notation $\set_{\overline{\psi}}(b, \overa):=\set(b, \overa)$ stands for the following set $$\big\{ \alpha \in u: \alpha < \alpha_{*} \ \text{and } \exists \overc \in \varphi(M, \overa) \ \text{s.t.} \  ( {\color{red}  \overb\overc\in\psi_{\alpha}(M)} \Leftrightarrow \alpha \in u) \big\}.$$\footnote{originally: $\set(b, \overa) = \{ \alpha \in u: \alpha < \alpha_{*} \ \text{and there is} \ \overc \in \varphi(M, \overa) \ \text{such that} \ \forall \alpha < \alpha_{*}(\psi_{\alpha}[\overc, \overa] \Leftrightarrow \alpha \in u) \}.$}

	iii)  By  $\Set_{\overline{\psi}}(\overa) :=\Set(\overa) $ we mean $$ \big\{ u \subseteq \alpha_{*}: \text{for some} \ \overb \in \varphi(M, \overa) \ \text{we have} \ u = \set(b, \overa) \big\}.$$

	iv)  By $\inter(\overline{a}) := \inter_{\overline{\psi}}(\overline{a})$ we mean $$\bigg\{ (w_{0}, w_{1}): w_{0} \subseteq w_{1} \subseteq \alpha_{*} \ \text{and } \exists\ u_{0}, u_{1} \in \Set(\overa) \ \text{s.t.} \ w_{1} \subseteq u_{1} \ \text{and} \ u_{0} \cap w_{1} = w_{0}  \bigg\}.$$
	\fi
\end{definition}
The main result of this section is the following.
\begin{proposition}
	\label{m8}
Adopt the previous notation. Then	$\affine_{2} \subseteq \affine_{1}.$
\end{proposition}
\begin{proof}
	We prove the theorem in a sequence of claims. We proceed by induction on the complexity of the formula $\vartheta$ that if $\vartheta \in \affine_{2},$
	then $\vartheta \in \affine_{1}$. This is clear if $\vartheta$ is an atomic formula. Suppose $\vartheta=\bigwedge_{i \in I}\vartheta_i$, and the claim holds for all $\vartheta_i , i \in I$. It is then clear from inductive step  that
	$\vartheta \in \affine_{1}$ as well. Similarly, if $\vartheta=\exists \overline{x} \varphi(\overline{x}),$ and if the claim holds for $\varphi(\overline{x})$,
	then clearly it holds for $\vartheta$.

	Now suppose that  $\alpha_*$ is an ordinal,  $\varphi(\overline{x}, \overline{y}), \langle \psi_{\alpha}(\overline{x}, \overline{y}): \alpha < \alpha_* \rangle$ are in $\affine_{2}$, such that  $\varphi(\overline{x}, \overline{y}) \geq \psi_{\alpha}(\overline{x}, \overline{y})$  for all $\alpha < \alpha_*.$  Also, suppose that
	$$\Upsilon \subseteq \{ (w_{0}, w_{1}): w_{0} \subseteq w_{1} \subseteq \alpha_{*} \}.$$
	Assume by the induction hypothesis that the formulas  $\varphi(\overline{x}, \overline{y})$ and $\psi_{\alpha}(\overline{x}, \overline{y})$, for $\alpha < \alpha_*$,
	are in $\affine_1.$ We have to show that
	$\vartheta(\overline{y}) = \Theta_{\varphi, \overline{\psi}, \Upsilon}(\overline{y}) \in \affine_{1}$ as well.

	Now, we bring the following claim:
	\begin{claim}\label{m11}
		Adopt  the above notation.
		Assume $\overline{a}_{l} \in {}^{\rm{lg(\overline{y})}}M$ for $l = 0, 1, 2 ,3$ and $\overline{a}_{3} = \overline{a}_{0} - \overline{a}_{1} + \overline{a}_{2}.$
		If $u_{j} \in \rm{Set}_{\varphi, \overline{\psi}}(\overline{a}_{j})$ for $j =0, 1$ and $u = u_{0} \cap u_{1}$,  then
		$$ \big\{  w \cap u: w \in \Set_{\varphi, \overline{\psi}}(\overline{a}_{2}) \big\} =\big \{ w \cap u: w \in \Set_{\varphi, \overline{\psi}}(\overline{a}_{3}) \big\}.$$
	\end{claim}
	\begin{proof}
		Let $\overline{b}_{0} \in \varphi(M, \overline{a}_{0})$ and $\overline{b}_{1} \in \varphi(M, \overa_{1})$ be such that
		$u_0= \set_{\overline{\psi}}(\overb_{0}, \overa_{0})$ and $u_1= \set_{\overline{\psi}}(\overb_{1}, \overa_{1})$.
		Suppose that $w \in \Set_{\varphi, \overline{\psi}}(\overline{a}_{2})$. Then for some $\overline{b}_{2} \in \varphi(M, \overline{a}_{2})$
		we have $w=\set_{\overline{\psi}}(\overb_{2}, \overa_{2})$. We have to find $w' \in \Set_{\varphi, \overline{\psi}}(\overline{a}_{3})$
		such that $w' \cap u=w \cap u.$
		
		Set $\overb_{3} := \overb_{0} - \overb_{1} + \overb_{2}$. Note that
		\[
		l=0, 1, 2 \implies ~ \overline{b}_{l} ^{\frown} \overline{a}_{l} \in \varphi(M),
		\]
		hence, as $\varphi \in \affine_1,$ we have
		\[
		\overline{b}_{0} ^{\frown} \overline{a}_{0} - \overline{b}_{1} ^{\frown} \overline{a}_{1} + \overline{b}_{2} ^{\frown} \overline{a}_{2} \in \varphi(M).
		\]
		Clearly, $$\overb_3 ^{\frown} \overa_3=\overline{b}_{0} ^{\frown} \overline{a}_{0} - \overline{b}_{1} ^{\frown} \overline{a}_{1} + \overline{b}_{2} ^{\frown} \overline{a}_{2}\in \varphi(M).$$
		According to its definition, $\overb_3 \in \varphi(M, \overa_3).$
		
		Let $w'=\set_{\overline{\psi}}(\overb_{3}, \overa_{3})$. We show that $w' \cap u=w \cap u.$ Suppose $\alpha \in u.$ Then,
		we have $\alpha \in u_0 \cap u_1$, and hence $\overb_j ^{\frown} \overa_j \in \psi_\alpha(M)$, for $j=0, 1$. Thus as $\psi_\alpha \in \affine_1$, we have

		$$\begin{array}{ll}
		\alpha \in w' & \iff    \overb_3 ^{\frown} \overa_3 \in \psi_\alpha(M)  \\
		& \iff \overb_2 ^{\frown} \overa_2 \in \psi_\alpha(M)\\
		& \iff\alpha \in w.
		\end{array}$$
		
		Suppose
		$w' \in \Set_{\varphi, \overline{\psi}}(\overline{a}_{3}).$
		By symmetry, $w\cap u=w' \cap u$ for some $w \in \Set_{\varphi, \overline{\psi}}(\overline{a}_{2}).$  The claim follows.
		\iffalse
		$\rm{set}(b_{0}, \overa_{0}) \supseteq u \subseteq \rm{set}(\overb_{1}, \overa_{2}). $ We claim:
		
		\begin{enumerate}
			\item[$(*)_{1}$] if $\overb_{2} \in \varphi(M, \overa_{2})$ and $\overb_{3} = \overb_{0} - \overb_{1} + \overb_{2}$ and $\alpha \in u$ then $$\alpha \in \rm{set}(b_{2}, \overa_{2}) \iff \alpha \in \rm{set}(b_{3}, \overa),$$
			
			\item[$(*)_{2}$] if $\overb_{3} \in \varphi(M, \overa_{3})$ and $- \overb_{2} = \overb_{0} - \overb_{1} - \overb_{3}$ and then the above hold.
		\end{enumerate}
		
		\smallskip
		Clearly $(*)_{1}$ and $(*)_{2}$ suffice and by symmetry $(*)_{1}$ suffice.
		
		% page 5.7
		
		To prove $(*)_{1}$ we assume $\rm{set}(b_{l}, \overa_{l}) \supseteq u$ for $l = 0, 1$ and we should prove $$\rm{set}(b_{\varp}, \overa_{\varp}) \cap u = \rm{set}(b_{2}, \overa_{2}) \cap u.$$

		First, let $\alpha \in u \cap \rm{set}(b_{2}, \overa_{2})$ and recall $b_{l} \in \psi_{\alpha}(M, \overa_{l})$ for $l = 0, 1$ by assumption of $(*)_{1}.$ So,
		
		\begin{itemize}
			\item $b_{l} \overa_{l} \in \psi_{\alpha}(M)$ for $l = 0, 1. $ Also,
			
			\item{\color{blue} $b_{0} \overa_{0} - b_{1} \overa_{1} + b_{2} \overa_{2} = (b_{0} - b_{1} + b_{1}) (\overa_{0} - a_{1} + a_{2}) = b_{3} \overa_{3},$}
		\end{itemize}
		so by the assumption on $\psi$ we have $  b_{3} \overa_{3} \in \psi_{\alpha}(M)$. {\color{red}This  means $\alpha \in \rm{set} (b_{3}, \overa_{2}).$}
		\noindent

		Second, assume $\alpha \in u \cap \rm{set}(b_{3}, \overa_{1}).$ So $(*)_{1}$ holds indeed, hence clause (a) holds.
		\fi
	\end{proof}
	Let us apply the previous claim and observe that:
	\begin{claim}\label{m17}
		Let $\varp=\len(\overline{y}) < \theta$ and  $\overa_{l} \in {}^{\varp}M$
		for
		$l = 0, 1, 2$,  and  set $\overa_{3} := \overa_{0} - \overa_{1} + \overa_{2}$.  If $\Upsilon\subseteq \bigcap_{l \leq 2} \inter_{\varphi, \overline{\psi}}(\overa_{l}) $, then  $\Upsilon \subseteq{\inter_{\varphi, \overline{\psi}}(\overa_{3})}.$
	\end{claim}
	
	\begin{proof}
		Let $(w_{0}, w_{1}) \in \Upsilon$. We shall prove that $(w_{0}, w_{1}) \in \inter(\overa_{3}).$ For $j \leq 2,$ as $(w_{0}, w_{1}) \in \Upsilon \subseteq \inter_{\varphi, \overline{\psi}}(\overa_{j})$,  there is a pair $u_{j, 0}, u_{j, 1} \in \set_{\varphi, \overline{\psi}}(\overa_{j})$ witnessing it. Namely, we have
		\[
		w_1 \subseteq u_{j, 1} \text{~and~} u_{j, 0} \cap w_1=w_0.
		\]
		Now, we can find $\overb_{j, 0}, \overb_{j, 1}$  such that $\rm{set}(\overb_{j, 0}, \overa_{j}) = u_{j, 0}$ and $ \rm{set}(\overb_{j, 1}, \overa_{j}) = u_{j, 1}.$
		Set
		\begin{enumerate}     \item[$\bullet$] $\overb_{3, 0} := \overb_{0, 0} - \overb_{1, 0} + \overb_{2, 0}$,
			\item[$\bullet$] $\overb_{3, 1} := \overb_{0, 1} - \overb_{1, 1} + \overb_{2, 1}.$
		\end{enumerate}
		
		By the argument of   Claim \ref{m11}, one may find some $u_{3, 1} \in \set_{\varphi, \overline{\psi}}(\overb_{3})$  and $u_{3, 0} \in \set_{\varphi, \overline{\psi}}(\overb_{3})$   such that
		the following two equalities are valid:
		\begin{enumerate}
			\item   $u_{3, 1} \cap (u_{0, 1} \cap u_{1, 1}) = u_{2, 1} \cap (u_{0, 1} \cap u_{1, 1})$, and
			\item   $u_{3, 0} \cap (u_{0, 0} \cap u_{1, 0}) = u_{2, 0} \cap (u_{0, 0} \cap u_{1, 0})$.
		\end{enumerate}
		By clause (1), we have $u_{3, 1} \supseteq w_1$.
		Hence
		$$\begin{array}{ll}
		w_0   &\subseteq u_{3, 1} \cap w_1\\&=  u_{3, 1} \cap (u_{0, 1} \cap u_{1, 1})\cap w_1\\
		&= \big(u_{3, 1} \cap (u_{0, 1} \cap u_{1, 1}) \big) \cap (u_{0, 1} \cap u_{1, 1})\cap w_1 \\
		&\stackrel{(2)}=\big( u_{2, 1} \cap (u_{0, 1} \cap u_{1, 1})\big) \cap (u_{0, 1} \cap u_{1, 1})\cap w_1\\
		&\subseteq w_0,
		\end{array}$$
		and so
		\[
		u_{3, 1} \cap w_1 = w_0.
		\]
		The claim follows.
	\end{proof}
	Now, we are ready to complete the proof of
	Proposition \ref{m8}.
	To this end, we fix the following data:
	
	\begin{enumerate}     \item[$\bullet_{1}$] $\vartheta(\overline{y}) = \theta_{\varphi, \overline{\psi}, \Upsilon}(\overline{y}),$
		\item[$\bullet_{2}$] $\overa_0, \overa_1, \overa_2 \in \vartheta(M)$,
		
		\item[$\bullet_{3}$] $\overa_3:=\overa_0 - \overa_1+\overa_2$.
		
	\end{enumerate}
	This gives us $\Upsilon \subseteq \inter_{\varphi, \overline{\psi}}(\overa_{l})$  for $l \leq 2$. Thanks to Claim \ref{m17},
	we know  $\Upsilon \subseteq{\inter_{\varphi, \overline{\psi}}(\overa_{3})}.$ According to Definition \ref{m5}(c) one has $\overa_3 \in \vartheta(M)$. Consequently,
	$\vartheta(\overline{y}) \in \affine_1$, and the proposition follows.
\end{proof}

\section{Additive frames}

In this section we  introduce the concept of an additive frame. Each additive frame contains, among other things, an abelain group. We will show that each abelian group can be realized in this way. In particular, the main result of this section is Theorem \ref{L7}.

The following is one of  our main and new frameworks:

\begin{definition}\label{k2}
\begin{enumerate}
\item[(A)]   We say
    $$\bff:=(M_\bff, \cL_\bff, \lambda_\bff, \kappa_\bff, \theta_\bff,\Omega_\bff)=(M, \cL, \lambda, \kappa, \theta,\Omega)$$
    is a \emph{general frame} if:
      \begin{enumerate}[(1)]
    	\item $M$ is a  $\tau_{M}$-model.
     %(we may write $M$ instead of $K = \{  M \}$); (we may use a class of  ``additive $\tau_{\kappa}$-models'' instead $\AB$),

	\item $\cL$ is a class or set of formulas in the vocabulary $\tau_{M},$ such that each $\varphi \in \cL$ has the form $\varphi(\overline{x}),  \ \overline{x}$ of length $< \theta$. %(the main case has been $\theta = \aleph_{0},$ will see if continue to be),
  %	\item If $\varphi(\overline{x}) \in \cL$,then  $\varphi(M) \subseteq {}^{\rm{lg (\overline{x})}} M$ is a submodel of $M,$

	\item For every $\overline{a} \in {}^{\varp}M, ~ \varp < \theta$, there is a formula
$ \varphi_{\overline{a}}(\overline{x})  \in \cL$
 such that:
        \begin{enumerate}[(a)]
            \item $\overline{a} \in  \varphi_{\overline{a}}(M),$

            \item (the minimality condition) if  $\psi(\overline{x}) \in \cL$ and  $\overline{a} \in \psi(M),$ then  $ \varphi_{\overline{a}}(M) \subseteq \psi(M).$
        \end{enumerate}

	\item

      \begin{enumerate}[(a)]
            \item If $\varphi_{\alpha}(\overline{x}) \in \cL$ for $\alpha < \kappa$, then for some $\alpha < \beta < \kappa$, we have $\varphi_{\alpha}(M) \supseteq \varphi_{\beta}(M),$

            \item if $\varphi_{\alpha, \beta}(\overline{x}, \overline{y}) \in \cL$ for $\alpha < \beta < \lambda$, then for some $\alpha_{1} < \alpha_{2} < \alpha_{3} < \lambda$ we have $\varphi_{\alpha_{1}, \alpha_{2}}(M) \supseteq \varphi_{\alpha_{1}, \alpha_{3}}(M), \varphi_{\alpha_{2}, \alpha_{3}}(M)$.
                 %or just among $\{ \varphi_{\alpha_{1}, \alpha_{2}}: \alpha_{1} < \alpha_{2} < \lambda \}$ there is a maximal,
      \end{enumerate}
       \item $\lambda, \kappa$ are regular and $\lambda \geq \kappa \geq \theta\geq |\Omega|+|\tau_M|,$  where $\Omega$ is a set of cardinals such that $1 \in \Omega$
       and all other cardinals in it are infinite.
 \end{enumerate}
\item[(B)] We say a general frame
$\bff$
is an \emph{additive frame} if in addition, it
satisfies:
\begin{enumerate}
\item[(6)]  $(\vert M \vert, +^{M}, -^{M}, 0^{M})$ is an abelian group. Moreover, $M$ is an additive $\theta$-model.

\item[(7)] If $\varphi(\overline{x}_u) \in \cL$, then  $\varphi(M) \subseteq {}^{u} M$ is a subgroup.
 \end{enumerate}

 \item[(C)] An additive frame $\bff$ is an  \emph{additive$^+$ frame} if $M_{\bff}$ has cardinality
 greater or equal to $\lambda$.
\end{enumerate}
\end{definition}
\begin{remark}
Given a general frame $\bff$ as above, we always assume that the language $\cL$ is closed under permutation of variables, adding dummy variables and finite conjunction.
\end{remark}
The next lemma is a criterion for an additive frame to be additive$^+$.
\begin{lemma}
\label{lem3} Suppose $\bff=(M, \cL, \lambda, \kappa, \theta,\Omega)$ is an additive frame. Then it is an additive$^+$ frame if and only if for each $\varepsilon \in (0, \theta)$,
there exists some $\bar a \in$$^{\varepsilon}M$ such that $\varphi_{\bar a}(M)$ has cardinality $\geq \lambda$.
\end{lemma}
\begin{proof}
The assumption clearly implies $\bff$ is additive$^+$. To see the other direction, suppose $\bff$ is an additive$^+$ frame
and $\varepsilon \in (0, \theta).$ Suppose by the way of contradiction,  $|\varphi_{\bar a}(M)| < \lambda$ for all $\bar a \in$$^{\varepsilon}M$. By induction on $\alpha < \kappa$
we can find a sequence $\langle  \bar a_\beta: \beta < \kappa    \rangle$ such that for each $\beta <\kappa,$
$\bar a_\beta \notin \bigcup_{\alpha < \beta }\varphi_{\bar a_\alpha}(M)$. This contradicts Definition \ref{k2}(4)(a).
\end{proof}
%\begin{convention}
%If $\kappa = \lambda$, we may omit $\kappa.$
%\end{convention}
%\begin{definition}\label{L2}
%	We say $\bff = (M, \cL, \lambda, \kappa, \theta)$ is nice \when \ it is as in Definition \ref{k2} (B) fixing $\langle \varphi_{\overline{a}}: %\overline{a} \in {}^{\theta > }M \rangle.$
%\end{definition}
The following defines a partial order relation on formulas of a frame.
\begin{definition}
\label{k86}
Assume $\bff$ is a general frame, and let  $\psi(\overline{x}), \varphi(\overline{x})$ be in $\cL_{\bff}$.
 \begin{enumerate}
\item  We say $\psi(\overline{x}) \leq \varphi(\overline{x})$  if $\psi(M) \subseteq \varphi(M).$\item We   say
    $\psi(\overline{x})$ and  $\varphi(\overline{x})$ are equivalent, denoted by $\psi(\overline{x}) \equiv \varphi(\overline{x})$, if $\varphi(M)= \psi(M).$

\item Suppose $\overline{a}, \overline{b} \in {}^{\varp}M$. We let $\overline{a} \leq \overline{b}$ (resp. $\overline{a} \equiv \overline{b}$)
 if $\varphi_{\overline{a}} \leq \varphi_{\overline{b}}$ (resp. $\varphi_{\overline{a}} \equiv \varphi_{\overline{b}}$).   We say $\overline{a}$ is equivalent with $\overline{b}$ if  $\overline{a} \equiv \overline{b}$.
\end{enumerate}
\end{definition}
\iffalse
\begin{convention}\label{k3}
    If not clear otherwise, $\bff$ is an additive frame, $K = K_{\bff}$ etc, and $M$ is a member of $K$ but we may ``forget'' to mention $M$ in   the $\varphi_{\overa}^{M}, -s.$
\end{convention}

\begin{observation}\label{k5}
    Without loss of generality $\varphi_{\overline{a}} = \varphi_{- \overline{a}}$ and $\varphi_{\overline{a}} \equiv \varphi_{\overline{b}} \Rightarrow \varphi_{\overline{a}} = \varphi_{\overline{b}}.$
\end{observation}
\fi

\begin{notation}\label{k8}
Assume $\bff=(M,  \lambda, \kappa, \theta,\Omega)$ is an additive frame.
Let $\overline{a}_{l}=\langle a_{l, \zeta}: \zeta < \varp   \rangle \in {}^{\varp}M$ for $l < n$. We set:

        \begin{itemize}

            \item  $-\overline{a}_{l} := \langle -a_{l, \zeta}: \zeta < \varp \rangle,$

            \item $\overline{a}_{1} + \overline{a}_{2} := \langle a_{1, \zeta} + a_{2, \zeta}: \zeta < \varp \rangle,$

            \item $\sum_{l < n} \overline{a}_{l} := \langle \sum_{l<n }a_{l, \zeta}: \zeta < \varp \rangle,$

            \item $\overline{a}-\overline{b}:=\overline{a}+(-\overline{b}).$
        \end{itemize}
\iffalse

    \sn
    1) $\varphi, \psi, \vartheta$ vary on $\mathscr{L},$

    \sn
    2) $\varp, \zeta, \xi$  vary on $\theta$  and $\cL_{\varp} = \{ \varphi(\overline{x}) \in \cL: \rm{lg} (\overline{x}) = \varp  \},$

    \sn
    3) $\overline{a}, \overline{b}, \overline{c}, \overline{d}$ vary on ${}^{\theta >} M $ and $M \in K,$

    \sn
    4) Let $\overline{a} \sim \overline{b}$ or ``$\overline{a}, \overline{b}$ are equivalent'' means $\rm{lg} (\overline{a}) = \rm{lg} (\overline{b})$ and $\varphi_{\overline{a}} \equiv \varphi_{\overline{b}}$ when $M$ is clear from the context; if not, we may say $\equiv_{M}$ or $\varphi^{M}_{\overa} = \varphi_{\overb}^{M}$ or e.g. (in M),

\begin{enumerate}
  \item \item Let $\varphi = \varphi(\overline{x}, \overline{y}) \in \cL,$ and $ \overline{a} \in {}^{\rm{lg}(\overline{x})}M$. By
  $\varphi(\overline{a}, M) $ we mean $$ \{  \overline{b} \in {}^{\rm{lg}(\overline{y})}M: M \models \varphi[ \overline{a}, \overline{b}]  \}.$$

   \item Let $\overline{a}_{l}=\langle a_{l, \zeta}: \zeta < \varp   \rangle \in {}^{\varp}M$ for $l < n$. We set:

        \begin{itemize}

            \item  $-\overline{a}_{l} := \langle -a_{l, \zeta}: \zeta < \varp \rangle,$

            \item $\overline{a}_{1} + \overline{a}_{2} := \langle a_{1, \zeta} + a_{2, \zeta}: \zeta < \varp \rangle,$

            \item $\sum_{l < n} \overline{a}_{l} := \langle \sum_{l<n }a_{l, \zeta}: \zeta < \varp \rangle,$

            \item $\overline{a}-\overline{b}:=\overline{a}+(-\overline{b}).$
        \end{itemize}
        \end{enumerate}
        \fi
\end{notation}

% page 2.4

\begin{lemma}\label{k11}
    Suppose $\bff=(M,  \lambda, \kappa, \theta,\Omega)$ is an additive frame, $\overline{a} \in {}^{\varp}M, \varphi = \varphi_{\overline{a}}$ and $\overline{a}_{\alpha} \in \varphi(M)$ for $\alpha < \lambda$.
     Then for some $\overline{\beta}, \overline{\gamma}$ we have:

    \begin{enumerate}[(a)]
        \item $\overline{\beta} = \langle \beta_{i}: i < \lambda \rangle \in {}^{\lambda} \lambda$ is increasing,

        \item $\overline{\gamma} = \langle \gamma_{i}: i < \lambda \rangle \in {}^{\lambda}\lambda $ is increasing,

        \item $\beta_{i} < \gamma_{i} < \beta_{i+1},$ for all $i<\lambda$,

        \item $\overline{a} - \overline{a}_{\beta_{i}} + \overline{a}_{\gamma_{i}}$ is equivalent to $\overline{a},$ for all $i<\lambda.$
    \end{enumerate}
\end{lemma}

\begin{proof}
    First, we reduce the lemma to the following claim:
    \begin{enumerate}
    \item[$(*)$] It is suffice to prove, for each sequences $\overline{a}, \langle \overline{a}_{\alpha}: \alpha < \lambda \rangle$ as above,  there are $\beta < \gamma < \lambda$ such that $\overline{a} - \overline{a}_{\beta} + \overline{a}_{\gamma}$ is equivalent to $\overline{a}.$
\end{enumerate}
To see this,   suppose  $(*)$ holds. By induction on $i<\lambda,$ we define the increasing sequences $\langle \beta_{i}: i < \lambda \rangle $  and $\langle \gamma_{i}: i < \lambda \rangle  $ as requested. Thus suppose that $i<\lambda$, and we have defined
$\langle \gamma_{j},\beta_{j}: j<i  \rangle  $. In order to define $(\beta_{i}, \gamma_{i})$,  we let $$\alpha_{*} := \sup \{ \gamma_{j} + \beta_{j} + 1: j < i  \}.$$
Since $\lambda$ is regular, $\alpha_{*}<\lambda$. Now, apply $(\ast)$ to $\overline{a}$ and $\langle \overline{a}_{\alpha_{*} + \alpha}:\alpha < \kappa\rangle$. This gives us
$\beta < \gamma < \kappa$ such that $$\overline{a} - \overline{a}_{\alpha_{*} +\beta} + \overline{a}_{\alpha_{*} +\gamma} \equiv \overline{a}.$$
Thus it suffices to set $\beta_i=\alpha_*+\beta$ and $\gamma_i=\alpha_*+\gamma.$

So, things are reduced in showing $(\ast)$ holds.
To see this, we define the formula $\varphi_{ \beta \gamma}$ as
$$(+)\quad\quad\quad \varphi_{ \beta \gamma}  := \varphi_{\overline{a} - \overline{a}_{\beta} + \overline{a}_{\gamma}}, $$
where $\beta < \gamma < \lambda$.
Note that $\overline{a},\overline{a}_{\beta} , \overline{a}_{\gamma}\in \varphi_{\overline{a}}(M),$
hence as $\varphi_{\overline{a}}(M)$ is a subgroup,
$\overline{a}-\overline{a}_{\beta}+ \overline{a}_{\gamma}\in \varphi_{\overline{a}}(M).$
Thanks to the minimality condition  from Definition \ref{k2}(3)(b), this implies that
 $\varphi_{\overline{a}} \geq \varphi_{\beta, \gamma}.$
Thus,
it is sufficient   to find $\beta < \gamma < \kappa $
 such that $ \varphi_{\overline{a}} \leq \varphi_{\beta, \gamma}$.
 By the property presented in Definition
 \ref{k2}(4)(b) there are  $\alpha_{1} < \alpha_{2} < \alpha_{3} < \kappa$
 such that  $\varphi_{\alpha_{1}, \alpha_{2}} \geq \varphi_{\alpha_{1}, \alpha_{3}}, \varphi_{\alpha_{2}, \alpha_{3}}.$
So,     \begin{enumerate}[(1)]
	\item $\overline{a} - \overline{a}_{\alpha_{1}} + \overline{a}_{\alpha_{2}} \in \varphi_{\alpha_{1}, \alpha_{2}}(M)$,

\item $\overline{a} - \overline{a}_{\alpha_{1}} + \overline{a}_{\alpha_{3}} \in
\varphi_{\alpha_{1}, \alpha_{3}}(M) \subseteq \varphi_{\alpha_{1}, \alpha_{2}}(M)$,

\item $\overline{a} - \overline{a}_{\alpha_2} + \overline{a}_{\alpha_3} \in \varphi_{\alpha_{2}, \alpha_{3}}(M) \subseteq \varphi_{\alpha_{1}, \alpha_{2}}(M).$

\end{enumerate}
Hence     \begin{equation*}
\begin{array}{clcr}
\overline{a}&=(1)-(2)+(3) \in\varphi_{\alpha_{1}, \alpha_{2}}(M).
%\\&\stackrel{(+)}=\varphi_{\overline{a} - \overline{a}_{\alpha_{1}} + \overline{a}_{\alpha_{2}}}(M).
\end{array}
\end{equation*}
Combining this with the minimality property from Definition \ref{k2}(3)(b) we observe that $\varphi_{\overline{a}} \leq \varphi_{\alpha_1, \alpha_2}$.
Thus it suffices to take $\beta=\alpha_1$
and $\gamma=\alpha_2$.
 \iffalse
 $$\varphi_{\overline{a}}(M) = \varphi_{\overline{a} - \overline{a}_{\alpha_{1}} + \overline{a}_{\alpha_{2}}}(M).$$In other words, $\overline{a}$ and $\overline{a} - \overline{a}_{\alpha_{1}} + \overline{a}_{\alpha_{2}}$ are equivalent as promised.
 \fi
\end{proof}

\begin{corollary}\label{k14}
 Suppose $\bff=(M,  \lambda, \kappa, \theta,\Omega)$ is an additive frame. The following assertions are valid:
\begin{enumerate}
\item Suppose $\varphi_{\overline{a}}(M)$ has cardinality $\geq \lambda$. Then there is $\overline{b} \in \varphi_{\overline{a}}(M)$ such that $\overline{b} \neq \overline{a}$ and $ \overline{b}$ is equivalent to $\overline{a}.$
 %   Similarly in $(M, \overline{c}) $ for $  \overline{c} \in {}^{\omega > }M.$

\item If for some $\overline{c}\in$$^{\varepsilon}M$,  $ \varphi_{\overline{c}}(M)$ has cardinality $\geq \lambda,$ then  the set
$
\{\overline{b} \in \varphi_{\overline{c}}(M): \overline{b} \text{~is equivalent to~} \overline{c}    \}
$
has cardinality $\geq \lambda.$

\item If $\bff$ is an additive$^+$ frame and $\varepsilon \in (0, \theta)$, then for some $\overline{a} \in$$^{\varepsilon}M$, the set
$\{\overline{b} \in \varphi_{\overline{a}}(M): \overline{b} \text{~is equivalent to~} \overline{a}    \}$ has cardinality $\geq \lambda.$
\end{enumerate}
\end{corollary}

\begin{proof}
(1) Since	$\varphi_{\overline{a}}(M)$ has cardinality $\geq \lambda$, we can take a sequence $\langle \overline{a}_{\alpha}: \alpha < \lambda\rangle$ of length $\lambda$ of pairwise distinct elements of $\varphi_{\overline{a}}(M)$  with no repetition. We apply Lemma \ref{k11}
	to find increasing sequences $\overline{\beta} = \langle \beta_{i}: i < \lambda \rangle$  and $\overline{\gamma} = \langle \gamma_{i}: i < \lambda \rangle  $ such that for all $i<\lambda$
		\begin{enumerate}[(a)]
		\item $\beta_{i} < \gamma_{i},$
		
		\item $\overline{a} - \overline{a}_{\beta_{i}} + \overline{a}_{\gamma_{i}}$ is equivalent to $\overline{a}.$
	\end{enumerate}
Set $\overline{b}_i:=\overline{a} - \overline{a}_{\beta_{i}} + \overline{a}_{\gamma_{i}}$. Since $\overline{a}_{\alpha}$'s are distinct,
 we deduce that $\overline{a} \neq \overline{b}_i$, for at least one $i<\lambda$. Thanks to (b), we know
$\overline{b}_i$ is equivalent to $\overline{a}.$

(2) Let $X= \{\overline{b} \in \varphi_{\overline{c}}(M): \overline{b} \text{~is equivalent to~} \overline{c}    \},$ and let $\mu=|X|.$ Suppose towards contradiction that $\mu < \lambda,$ and let
$\langle \overline{b}_i: i<\mu  \rangle$  enumerate $X$. Since
 $ \varphi_{\overline{c}}(M)$ has cardinality $\geq \lambda,$ let
$\langle \overline{c}_{\alpha}: \alpha < \lambda\rangle$ be a sequence of length $\lambda$ of pairwise distinct elements of $\varphi_{\overline{c}}(M)$  with no repetition. By induction on $\alpha < \lambda$
we can find $\xi_\alpha < \lambda$ such that
\begin{enumerate}
\item[$(*)_\alpha$:] $\overline{c}_{\xi_\alpha} \notin \{ \overline{b}_i - \overline{c} +\overline{c}_{\xi_\beta}: \beta < \alpha, i< \mu    \}$.
\end{enumerate}
By the argument of clause (1), applied to the sequence $\langle \overline{c}_{\xi_{\alpha}}: \alpha < \lambda\rangle$, we can find some $\beta_{\iota} < \gamma_{\iota}< \lambda$ such that
$\overline{c} - \overline{c}_{\xi_{\beta_{\iota}}} + \overline{c}_{\xi_{\gamma_{\iota}}}$ is equivalent   to  $\overline{c}.$   Thus for some
$i<\mu, \overline{c} - \overline{c}_{\xi_{\beta_{\iota}}} + \overline{c}_{\xi_{\gamma_{\iota}}}=\overline{b}_i$. But then
\[
 \overline{c}_{\xi_{\gamma_{\iota}}} = \overline{b}_i -  \overline{c} + \overline{c}_{\xi_{\beta_{\iota}}},
\]
which contradicts $(*)_{\gamma_{\iota}}$.

(3) By Lemma \ref{lem3} and clause (2).
	\end{proof}
\begin{lemma}\label{k17}
\begin{enumerate}
\item Suppose $\bff=(M,  \lambda, \kappa, \theta,\Omega)$ is a general frame,  $\varp < \theta$ and $\overline{a}_{\alpha} \in {}^{\varp}M$ for $\alpha < \kappa$. Then there is some $\alpha < \kappa$ such that the set $\{ \beta < \kappa: \overline{a}_{\beta} \in \varphi_{\overline{a}_{\alpha}} (M) \}$ is unbounded in $\kappa.$

    \item In clause (1), we can replace $\kappa$ by any cardinal $\kappa' \geq \kappa.$
 \end{enumerate}
\end{lemma}

\begin{proof}
We prove clause (2). Set $\varphi_{\alpha}:=\varphi_{\overline{a}_{\alpha}} $. Suppose
on the way of contradiction that, for each $\alpha < \kappa',$ the set
$X_\alpha=\{ \beta < \kappa': \overline{a}_{\beta} \in \varphi_{\alpha} (M) \}$ is bounded in $\kappa'.$  So,
\begin{enumerate}
\item[$(*)_1$] $\forall\alpha<\kappa', \exists \alpha<\beta_\alpha<\kappa'$ such that $\forall \beta\geq \beta_\alpha$ we have $ \varphi_{\alpha}\ngeq\varphi_{\beta}.$
\end{enumerate}
We define an increasing  and continuous sequence $\langle \zeta_\alpha:\alpha<\kappa \rangle $ of ordinals less than $\kappa'$,  by induction on $\alpha$ as follows:

\begin{itemize}
	
	\item  $\zeta_0:=0,$
	
	\item $\zeta_{\alpha+1}:=\beta_{\zeta_\alpha}$,
	
	\item $\zeta_{\delta}:=\lim_{\alpha<\delta}\zeta_\alpha$ for  limit ordinal $\delta$.
\end{itemize}
Consider  the sequence
$\{\varphi_{\zeta_\alpha}:\alpha<\kappa\}$,
and apply the property presented in Definition \ref{k2}(4)(a) to find $\gamma<\delta<\kappa$ such that
\begin{enumerate}
\item[$(*)_2$] $ \varphi_{\zeta_\gamma}\geq\varphi_{\zeta_\delta}.$
\end{enumerate}
Since $\gamma<\delta$,  $\zeta_\delta\geq \zeta_{\gamma+1}=\beta_{\zeta_\gamma}$. We apply $(*)_1$ for $\alpha:=\zeta_\gamma$ and $\beta:=\zeta_\delta\geq \beta_\alpha$. This gives us $ \varphi_{\zeta_\gamma}\ngeq\varphi_{\zeta_\delta}, $ which contradicts $(*)_2$.
\end{proof}

\iffalse
\subsection{On nice frame existence}\label{3}
\bigskip

\begin{definition}\label{L2}
    We say $\bff = (M,  \lambda, \kappa, \theta,\Omega)$ is nice \when \ it is as in \ref{k2} fixing $\langle \varphi_{\overline{a}}: \overline{a} \in {}^{\theta > }M \rangle.$ If $\kappa = \lambda$ we may omit $\kappa.$
\end{definition}

\begin{theorem}\label{L4}
    \

    \noindent
    1) For an $R$-module, $\varphi(\overline{x}_{[\varp]}) \in \mathbb{L}_{\infty, \theta}^{\rm{{cpe}}}, $ the set of cositive existential formulas (= the closure of the atomic formulas under $\bigwedge_{\alpha < \beta}$ and $(\exists \overline{x}_{\varp, \zeta })$ ($\varp < \zeta < \theta$)) (so no $\neg,$ no disjunction $\vee$), we have:

    \smallskip
    \begin{itemize}
        \item[$(*)$] $\varphi(M) = \{ \overline{a} \in {}^{\varp}M: M \models \varphi[\overline{a}] \}$ is an abelian subgroup of $ ({}^{\varp}M, +).$
    \end{itemize}

    \noindent
    2) We can use $\mathbb{L}^{\rm{cop}}$ when we allow $\forall \overline{x},$ ($\rm{cop}$ stand for no disjunction no negation).

\end{theorem}

\begin{PROOF}{\ref{L4}}
    Just check or quote or  see \S \ref{1}.
\end{PROOF}

\fi

In  what follows we need to use  a couple of results from \cite{Sh:110}. To make the paper more self contained, we borrow some definitions and results from it.
\begin{definition}
 \begin{enumerate}
	\item By a \textit{\ tree} \index{tree} we mean a partially-ordered set $(\cT,\leq )$ such that for all $%
t\in \cT$,  $\pred( t ):=\{s\in \cT:s<t\}, $ is a well-ordered set; moreover, there
is only one element $r$ of $\cT$, called the \textit{root} of $\cT$, such that
$\pred( r)$ is empty.

\item The order-type of $\pred(t)$ is called the height of $t$,
denoted by ht$(t)$.
\item The height of $\cT$ is $\sup \{$ht$(t) + 1:t\in \cT\}$. \end{enumerate}
\end{definition}

\begin{definition}
\label{quasi}
\begin{enumerate}
\item A \textit{quasi-order} $\mathcal{Q}$ is a pair $(\mathcal{Q},\leq _{\mathcal{Q}})$ where $\leq _{\mathcal{Q}}$ is a
reflexive and transitive binary relation on $\mathcal{Q}$.
	\item   $\mathcal{Q}$ is called $\kappa $\textit{--narrow}, if there is no
	\textit{antichain} in $\mathcal{Q}$ of size $\kappa $, i.e., for every $f:\kappa
	\rightarrow \mathcal{Q}$ there exist $\nu \neq \mu $ such that $f(\nu )\leq _{\mathcal{Q}}f(\mu )
	$.

	\item For a  quasi-order  $\mathcal{Q}$,  a $\mathcal{Q}$\textit{-labeled tree} is a pair $(\cT,\Phi
_{\cT}) $ consisting of a tree $\cT$ of height $\leq \omega $ and a function $%
\Phi _{\cT}:\cT\rightarrow \mathcal{Q}$.

\item $\mathcal{Q}$ is $\kappa$-well ordered if for every sequence
$\left\langle q_i: i<\kappa   \right\rangle $ of elements of $\mathcal{Q},$ there are $i<j<\kappa$ such that $q_i \leq_{\mathcal{Q}} q_j.$
\end{enumerate}
\end{definition}
\begin{remark}	On any set of $\mathcal{Q}$-labeled trees we define a
quasi-order by: $(\cT_{1},\Phi _{1})\preceq (\cT_{2},\Phi _{2})$ if and only if
there is a function $\nu :\cT_{1}\rightarrow \cT_{2}$ equipped  with the following properties:

\begin{itemize}
	
	\item[a)]  for all $t\in \cT_{1}$, $\Phi _{1}(t)\leq _{\mathcal{Q}}\Phi _{2}(\nu (t))$,
	
	\item[b)] $t\leq _{\cT_{1}}t^{\prime } \Longrightarrow \nu(t)\leq
	_{\cT_{2}}\nu(t^{\prime }),$
\item[c)] for all $t\in \cT_{1}$, $\Ht_{\cT_{1}}(t)=\Ht_{\cT_{2}}(\nu(t))$.
\end{itemize}

\end{remark}

\begin{definition}\label{beau}
\begin{enumerate}
\item Given infinite cardinals $\kappa$ and $\mu,$ the notation $\kappa \longrightarrow (\omega)_\mu^{<\omega}$ means that: for every
function $f: [\kappa]^{<\omega} \to \mu$, there exists an infinite
subset
	$X \subseteq \kappa$ and a function $g : \omega \to \mu$ such that $f(Y) = g(|Y|)$ for all finite subsets
	$Y$ of $X$.

\item	Let $\kappa_{\rm{beau}}$ denote the first beautiful\footnote{This also is called the first $\omega$-Erd\"{o}s cardinal.} cardinal. This is defined as the smallest cardinal $\kappa$ such that
	$\kappa \longrightarrow (\omega)_2^{<\omega}$.

\item Given a ring $R$ and an infinite cardinal $\theta$, let $\kappa_{\rm{beau}}(R, \theta)$ denote the least cardinal $\kappa$
such that $\kappa \longrightarrow (\omega)_{|R|+\theta^{<\theta}}^{<\omega}$.

\item
Given a vocabulary $\tau$, let $\kappa_{\rm{beau}}(\tau,\theta)$ denote the least cardinal $\kappa$
such that $\kappa \longrightarrow (\omega)_{|\tau|+\theta^{<\theta}}^{<\omega}$. If $\theta=\aleph_0$, we may omit it.
\end{enumerate}
 \end{definition}
Now, we can state:
\begin{fact}(Shelah, \cite[theorems 5.3+ 2.10] {Sh:110})
	\label{qwo}
	Let $\mathcal{Q}$ be a quasi-order of cardinality $<\kappa_{\rm{beau}}$, and $%
	\mathcal{S}$ be a set of $\mathcal{Q}$-labeled trees with $\leq \omega$ level. Then $\mathcal{S}$ is $\kappa_{\rm{beau}}$-narrow, and even
$\kappa_{\rm{beau}}$-well ordered.
\end{fact}

We  are now ready to state and prove the main result of this section.
\begin{theorem}\label{L7}
\begin{enumerate}
\item[(i)]   Assume $M$ is an $R$-module and $\kappa=\kappa_{\rm{beau}}(R, \theta)$ and $\Omega$ is such that $1 \in \Omega$ and $|\Omega| \leq \theta$. Also assume that
    $\lambda \geq \kappa$ is regular and satisfies  $\lambda \to (\kappa + 1)_{4}^{3}$. The
    following hold:

    	\begin{enumerate}[(a)]
    	\item $\bff = (M,  \lambda, \kappa, \theta,\Omega)$ is an additive frame, whenever
    	$$\cL \subseteq \{ \varphi(\overline{x}): \varphi \in \cL_{\infty, \theta}^{\rm{cop}}(\tau_{M}), \rm{lg}(\overline{x}) < \theta \}$$
    	is  closed under arbitrary conjunctions.
    	
    	\item  $\bff = (M,  \lambda, \kappa, \theta,\Omega)$ is an additive frame, whenever
    	$$\cL \subseteq \{ \varphi(\overline{x}): \varphi \in \cL_{\infty, \theta}^{\rm{co}}(\tau_{M}), \rm{lg}(\overline{x}) < \theta \}$$
    	is  closed under arbitrary conjunctions.
    \end{enumerate}
\item [(ii)]  	Let $M$ be a $\tau$-model, $\kappa=\kappa_{\rm{beau}}(\tau,\theta)$ and let $\lambda \geq \kappa$ be regular and satisfies  $\lambda \to (\kappa + 1)_{4}^{3}$. Then $\bff = (M,  \lambda, \kappa, \theta,\Omega)$ is a additive frame, whenever
$$\cL \subseteq \{ \varphi(\overline{x}): \varphi \in \cL_{\infty, \theta}^{\rm{co}}(\tau_{M}), \rm{lg}(\overline{x}) < \theta \}$$
is  closed under arbitrary conjunctions.

\item[(iii)] Suppose in addition to  (i) (resp.  (ii)) that $|M| \geq \lambda$. Then  $\bff$ is an additive$^+$ frame.
\end{enumerate}
    \iffalse
    i) $ \bff = (\mathscr{L}, \kappa, \lambda, M, \theta)$ is an additive frame \when:

    \begin{enumerate}[a)]
        \item $\mathscr{L} \subseteq \{ \varphi(\overline{x}_{[\varp]}): \varphi \in \mathbb{L}^{\rm{cos}}(\tau_{M}), \varp < \theta \},$

        \item $\mathscr{L}$ is closed under arbitrary conjunctions,

        \item $\kappa > \vert R \vert + \theta, \ \theta \geq \aleph_{0}, \ \kappa$ is a beautiful cardinal,

        \item $\lambda \geq \kappa$ is weakly compact or just $\lambda \to (\kappa + 1)_{4}^{3}$ and $\lambda$ regular e.g. $\lambda = (2^{2^{k}})^{+}$ (or less suffice? see \ref{L10}).
    \end{enumerate}

ii) The claims \ref{k11}, \ref{k14} holds for  $(\mathscr{L}, \lambda, \kappa, M, \theta).$
 \fi
\end{theorem}
\begin{proof}
 (i):   We have to show that $\bff$ satisfies the relevant items  of Definition \ref{k2}. Items (1), (2) and (5)
    are true by definition. As $M$ is an $R$-module,  clause (6) is valid.
  The validity of  (7)  follows from Lemma \ref{a5}.

Let us consider clause (3).
Thus suppose that   $\overline{a} \in {}^{\varp}M$, where  $\varp < \theta$. We are going to find some  $\varphi \in \cL$ such that:
    \begin{enumerate}[i)]
    	\item[$\bullet$] $\overline{a} \in \varphi(M),$
    	
    	\item[$\bullet$] if  $\psi \in \cL$ and  $\overline{a} \in \psi(M),$ then  $\varphi(M) \subseteq \psi(M).$
    \end{enumerate}
   For any $\overline{b} \in  {}^{\varp}M,$ if there is some formula
      	$\varphi(\overline{x})$ such that
      $$(\dagger)_1 \quad\quad\quad M\models\varphi[\overline{a}]\wedge \neg \varphi[\overline{b}],$$
      then let $\varphi_{\overline{b}}  \in \cL$ be such a formula. Otherwise, let
      $\varphi_{\overline{b}}  \in \cL$ be any  true formula such that
  $ \varphi_{ \overline{b}}(M)= {}^{\varp}M$.
Finally set
  $$(\dagger)_2 \quad\quad\quad\varphi:= \bigwedge \{ \varphi_{\overline{b}}: \overline{b} \in {}^{\varp}M  \}.$$
Now, we claim that $\varphi$ is as desired. First, we check (3)(a), that is $\overline{a} \in \varphi(M).$
As
 $\varphi(M)=\bigcap_{\overline{b} \in {}^{\varp}M}   \varphi_{\overline{b}}(M),$
  it suffices to show that  $\overline{a} \in \varphi_{\overline{b}}(M)$, for $\overline{b} \in {}^{\varp}M$.
  Fix $\overline{b}$ as above. If there is no formula $\varphi(\overline{x})$ as in $(\dagger)_1$, then
$\overline{a} \in {}^{\varp}M= \varphi_{ \overline{b}}(M)$, and we are done. Otherwise, by its definition, $M \models  \varphi_{ \overline{b}}(\overline{a})$, and hence, again we have $\overline{a} \in \varphi_{\overline{b}}(M)$.

 To see (3)(b) holds,  let $\psi(\overline{x}) $ be such that $ \overline{a} \in \psi(M).$ We have to show that  $\varphi(M) \subseteq \psi(M).$
 Suppose by the way of contradiction that $\varphi(M) \nsubseteq \psi(M).$
 Take $ \overline{b}\in\varphi(M)
     \setminus \psi(M)$. Now, $M \models \neg  \psi[\overline{b}],$ and by our assumption $M \models  \psi[\overline{a}].$ In particular, by our construction, the formula
    $\varphi_{\overline{b}}$ satisfies
    $$(\dagger)_3\quad\quad\quad M \models \varphi_{\overline{b}}[\overline{a}] \wedge \neg \varphi_{\overline{b}} [\overline{b}].$$

     Now,

      \begin{enumerate}
     	\item[$\bullet_{1}$]
     	 $\overline{b}\in\varphi(M)\stackrel{(+)}\subseteq\varphi_{\overline{b}}(M)$,
     	
     	\item[$\bullet_{2}$] by $(\dagger)_3$, $M \models  \neg \varphi_{\overline{b}} [\overline{b}]$, and hence $\overline{b}\notin \varphi_{\overline{b}}(M)$.
     \end{enumerate}
     By $\bullet_{1}$ and $\bullet_{2}$ we get a contradiction.

Now we turn to clause (4). First let us consider (4)(a). Thus suppose that
$\varphi_{\alpha}(\overline{x}) \in \cL$, for $\alpha < \kappa$, are given.
We should find some $ \alpha < \beta<\kappa $ such that $\varphi_{\alpha} \geq \varphi_{\beta}$, i.e., $\varphi_{\alpha}(M) \supseteq \varphi_{\beta}(M).$
To this end, first note that we can restrict ourselves to those formulas  such that both free and bounded variables appearing in them are among $\{x_i:i<\theta\}$, the set of free variables of
$\varphi$ has  the form $\{x_\zeta: \zeta  < \varepsilon \}$, and the quantifiers have the form
$ \exists \bar x_{[\varepsilon_0, \varepsilon_1)} \text{~and~} \forall \bar x_{[\varepsilon_0, \varepsilon_1)},$
 where $\varepsilon_0 < \varepsilon_1 <\theta,$ and $\bar x_{[\varepsilon_0, \varepsilon_1)}=\langle x_\xi: \varepsilon_0 \leq  \xi < \varepsilon_1      \rangle$.
 In what follows, writing a formula as $\varphi(\bar{x})$, we mean $\bar{x}$ lists the free variables appearing in $\varphi$ in increasing order.

We can consider a formula $\varphi(\bar{x})$ as a type $(\cT,\boldmath c)$ such that

  \begin{enumerate}[(a)]
 	\item $\cT$ is  a tree with $\leq \omega$ levels with no infinite branches,
 	
 	\item $\boldmath c$ is  a function with domain $\cT$,
 	
 	\item if $ t\in \cT\setminus \max(\cT)$,
 	then $c(t)$ is in the following set$$\bigg\{\wedge,\exists\bar{x},\forall\bar{x}:\bar{x} \emph{ has form }  \bar x_{[\varepsilon_0, \varepsilon_1)} \emph{ for some } \varepsilon_0 < \varepsilon_1<\theta\bigg\},$$
 	\item if $ t\in \max(\cT)$ then $\boldmath c(t)$ is an atomic formula in $\tau_M$.
 \end{enumerate}
Clearly,   $|\Rang(\boldmath c)|\leq\theta+|R|$. For each $\alpha < \kappa$ set
$\varphi_{\alpha}(\bar{x}):=(\cT_{\alpha},\boldmath c_{\alpha}).$

Let $\mathcal{Q}$ be the range of the function $c$. Then
it is a quasi-order under the $\leq$ relation, where the only relation is between atomic formulas in $\mathcal{Q}$, as  defined in Definition \ref{k86}, and clearly, it has cardinality
$
|\mathcal{Q}| \leq |\tau_M| + \theta^{<\theta} <\kappa_{\rm{beau}}.
$
In particular, by Fact \ref{qwo},  applied to the sequence
 $\big\langle (\cT_{\alpha},\boldmath c_{\alpha}): \alpha < \kappa \big\rangle,$
we get some    ${\alpha}<\beta$ and a function $f$ equipped with the following property:

\begin{enumerate}[(a)]
	\item[$(\ast)$] $f$ is a 1-1 function from  $\cT_{\alpha}$ into $\cT_{\beta}$, such that:
\begin{itemize}
\item $f$ is	 level preserving, i.e., for all $t\in \cT_{\alpha}, \Ht_{\cT_{\alpha}}(t)=\Ht_{\cT_{\beta}}(f(t))$,

\item $f$ is	order preserving, i.e., for all $t \leq_{\cT_{\alpha}}  s$, we have    $f(t) \leq_{\cT_{\beta}}  f(s),$

\item $f$ is	non-order preserving, i.e., for all incomparable $t, s$ in $\cT_{\alpha}$, we have    $f(t),f(s)$ are incomparable in $\cT_{\beta}$,

\item If $t\in \cT_{\alpha}\setminus\max(\cT_{\alpha})$, then $\boldmath c_{\alpha}(t) \leq \boldmath c_{\beta}(f(t)).$
 \end{itemize}
\end{enumerate}
Suppose $t\in \cT_{\alpha}\setminus\max(\cT_{\alpha})$. This in turns imply that $\boldmath c_{\alpha}(t)$ and $\boldmath c_{\beta}(f(t))$ are not atomic formulas, thus if
 $\boldmath c_{\alpha}(t) \neq \boldmath c_{\beta}(f(t))$, then $\boldmath c_{\alpha}(t)$ and $\boldmath c_{\beta}(f(t))$ are incomparable. Hence by the last bullet condition,
 $$t\in \cT_{\alpha}\setminus\max(\cT_{\alpha}) \Rightarrow \boldmath c_{\alpha}(t) = \boldmath c_{\beta}(f(t)).$$
Let  $\cT_{\alpha}':=\Rang(f)$ and $\boldmath c_{\alpha}':=\boldmath c_{\beta}\upharpoonright\cT_{\alpha}'$. By this notation, $\varphi_{\alpha}(\bar{x})$ can also be represented by $(\cT_{\alpha}',\boldmath c_{\alpha}')$. For any $t\in\cT_{\alpha}'$ we  define
$\varphi_t^1$ to be the formula represented by
$$\bigg(\{s\in\cT_{\alpha}':t\leq_{\cT_{\alpha}'}s\},\boldmath {c}_{\beta}\upharpoonright \{s\in\cT_{\alpha}':t\leq_{\cT_{\alpha}'}s\}\bigg),$$
and define $\varphi_t^2$ to be  the  formula
represented  by $$\bigg(\{s\in\cT_{\beta} :t\leq_{\cT_{\beta}'}s\},\boldmath c_{\beta}  \upharpoonright  \{s\in\cT_{\beta}:t\leq_{\cT_{\beta}'}s\}\bigg).$$
Note that the formula
$\varphi_t^1$ may have fewer free variables than $\varphi_t^2$, but we can add the remaining variables. So, we may and do  assume that  $\varphi_t^{\ell}=\varphi_t^{\ell}(\bar{x}_t)$  for $\ell=1,2.$
Now we are going to prove that for every $t$,
$
\varphi^2_t(M) \subseteq \varphi^1_t(M).$ Observe that this is enough as $\varphi^1_r=\varphi_\alpha$ and
$\varphi^2_r=\varphi_\beta$ where $r$ stands for the root of the tree.
This is done by induction on the depth of $t$ inside $\cT_{\alpha}'$
 which is possible, as  $\cT_{\alpha}'$ is well-founded.
By the way we defined our trees, it suffices to deal with the following  cases:
 \begin{enumerate}
	\item[]Case 1): $c_{\beta}(t)$ is a basic formula, i.e., an atomic formula or its negation.
	\item[]Case 2): $c_{\beta}(t)$ is $\wedge$.
	\item[]Case 3): $c_{\beta}(t)$ is $\exists \overline{x}'$ or just $c_{\beta}(t)$ is $\exists^{\sigma} \overline{x}'$.
	\item[]Case 4) $c_{\beta}(t)$ is $\forall \overline{x}'$ or just $c_{\beta}(t)$ is $\forall^{\sigma} \overline{x}'$.
\end{enumerate}
Let discuss each cases separably:

  Case 1): Here, $t$ is a maximal node of the tree $\cT_{\beta}$. Hence, necessarily a maximal node $
	\varphi^1_t=c_{\beta}(t)=\varphi^2_t.
$
Consequently, the conclusion is obvious.

 Case 2): Here, we have\begin{enumerate}
		\item[$\bullet$]
		$	\varphi^2_t:=\bigwedge\{\varphi^2_s:s\in\suc_{\cT_{\beta}}(t)\} $,
		
		\item[$\bullet$] $	\varphi^1_t:=\bigwedge\{\varphi^1_s:s\in\suc_{\cT_{\alpha}'}(t)\} $.
	\end{enumerate}
By the choice of the $\cT_{\alpha}'$ and the function $f$, we have  $\suc_{\cT_{\alpha}'}(t)\subseteq \suc_{\cT_{\beta}}(t).$ Now, due to the induction hypotheses$$s\in\suc_{\cT_{\alpha}'}(t)\Longrightarrow \varphi^2_s(M) \subseteq \varphi^1_s(M). $$ According to the definition of sanctification for $\wedge$ and
$\varphi^2_s$ we are done, indeed:

$$\begin{array}{ll}
\varphi^1_t(M) & = \bigcap\{\varphi^1_s(M):  s\in\suc_{\cT_{\alpha}'}(t)  \} \\
& \supseteq  \bigcap\{\varphi^2_s(M):  s\in\suc_{\cT_{\alpha}'}(t)  \}\\
& \supseteq \bigcap\{\varphi^s_s(M):  s\in\suc_{\cT_{\beta}}(t)  \}.
\end{array}$$

Case 3):
	Let $\overline{x}'_t=\overline{x}_s^{\frown}\overline{x}'$, and recall that
	$\suc_{\cT_{\beta}}(t)$ is singleton, and also  	 $\suc_{\cT_{\alpha}}(f^{(-1)}(t))$  is singleton, because $\boldmath c_{\beta}(t)=\boldmath c_{\alpha}(f^{-1}(t))$. This implies that $\suc_{\cT_{\alpha}'}(t)$  is singleton, say $\suc_{\cT_{\alpha}'}(t)=\{s\}=\suc_{\cT_{\beta}}(t)$. In order to see $\varphi^1_t(M) \supseteq \varphi^2_t(M)$, we take $\overa \in {}^{\lg(\overline{x})}M$ be such that $M\models\varphi^2_t[\overa]$ and shall  show
	that $M\models\varphi^1_t[\overa]$. Indeed,
	 $$\varphi^{\ell}_t(\overline{x}):=(\exists^{\sigma} \overline{x}')\varphi^{\ell}_s(\overline{x}_s,\overline{x}')\quad \ell=1,2,$$
 and since
$
 M \models (\exists^{\sigma} \overline{x}')\varphi^{2}_s(\overa,\overline{x}'),
$
 necessarily,  for some pairwise disjoint
	 $\overb_{\zeta} \in {}^{\lg(\overline{x}'}M$ one has $M\models\varphi^2_s[\overa,\overb_{\zeta}].$
	 Thanks to the inductive hypothesis, we know
	  $M\models\varphi^1_s[\overa,\overb_{\zeta}].$
	 According to the definition of sanctification, we have$$M\models (\exists^{\sigma} \overline{x}')\varphi^1_s(\overa,\overline{x}'),$$which means that $M\models\varphi^1_t[\overa]$, as promised.

 Case 4):
	Let $\overline{x}'_t=\overline{x}_s^{\frown}\overline{x}'$, and recall that
$\suc_{\cT_{\beta}}(t)$ is singleton, and also  	 $\suc_{\cT_{\alpha}}(f^{-1}(t))$  is singleton, because $\boldmath c_{\alpha}(f^{-1}(t))\subseteq\boldmath c_{\beta}(t)  $. This implies that $\suc_{\cT_{\alpha}'}(t)$  is singleton, say $\suc_{\cT_{\alpha}'}(t)=\{s\}=\suc_{\cT_{\beta}}(t)$. In order to see $\varphi^1_t(M) \supseteq \varphi^2_t(M)$, we take $\overa \in {}^{\lg(\overline{x})}M$ be such that $M\models\varphi^2_t[\overa]$ and shall  show
that $M\models\varphi^1_t[\overa]$. Similar to the Case (3), we can write
$$\varphi^{\ell}_t(\overline{x}):=(\forall^{\sigma} \overline{x}')\varphi^{\ell}_s(\overline{x}_s,\overline{x}')\quad \ell=1,2.$$
Suppose it is not the case that $M \models \varphi^1_t[\overa]$. This means, by Definition \ref{a21}(2)(b),
that there are pairwise disjoint
	$\overb_{\zeta} \in {}^{\lg(\overline{x}')}M$, for $\zeta<\sigma$ such that  $M\models\neg\varphi^1_s[\overa,\overb_{\zeta}].$
By the induction hypothesis, we have
$\varphi^1_s(M) \supseteq \varphi^2_s(M)$, hence  for each $\zeta<\sigma$,
$$M\models\neg\varphi^2_s[\overa,\overb_{\zeta}].$$
The later means  that $M \models \varphi^2_t[\overa]$ is not true, which contradicts our initial assumption.
This completes the proof of clause (4)(a).

Finally, let us turn to check the property presented in clause (4)(b) from Definition \ref{k2}.
Suppose   $\varphi_{\alpha, \beta}(\overline{x}) \in \cL$ for $\alpha < \beta < \lambda$ are given. We need to find some $\alpha_{1} < \alpha_{2} < \alpha_{3} < \kappa$ such that  $\varphi_{\alpha_{1}, \alpha_{2}} \geq \varphi_{\alpha_{1}, \alpha_{3}}, \varphi_{\alpha_{2}, \alpha_{3}}.$  To see this,
we  define a coloring $ \mathbf{c}: [\lambda]^{3} \to 2 \times 2$ as follows. Fix $\alpha < \beta < \gamma < \lambda$, and define
the following pairing function

  $$ \mathbf{c} (\{ \alpha, \beta, \gamma \}):= \bigg(\emph{truth value of }  \varphi_{\alpha, \beta} \geq \varphi_{\alpha, \gamma} ,  \emph{  truth value of }  \varphi_{\alpha, \gamma} \geq \varphi_{\beta, \gamma}  \bigg).$$
  By the assumption  $\lambda \to (\kappa + 1)_{4}^{3}$, there is $X \subseteq \lambda$ of order type $\kappa +1$ such that $\mathbf{c} \rest [X]^{3}$ is constant. Let $\alpha_{i} \in X$, for $i \leq \kappa$,  be  an increasing enumeration of $X$.  Consider the sequence
   $ \{\varphi_{\alpha_{0}, \alpha_{ i}}: i< \kappa\}, $
   and applying clause (4)(a) to it.  This gives us  $i < j < \kappa$ such that $\varphi_{\alpha_{0}, \alpha_{i}} \geq \varphi_{\alpha_{0}, \alpha_{j}}.$
  Note that this implies that $\mathbf{c} (\alpha_{0},\alpha_{i},\alpha_{j})= (1,\iota)$, for some  $\iota  \in \{0, 1\}.$
   Since   $\mathbf{c} \rest [X]^{3}$ is constant, it follows that  for any $\alpha < \beta < \gamma$   from $X$, we have
$\mathbf{c} (\alpha, \beta, \gamma)= (1,\iota)$, in particular
    \begin{itemize}
        \item[$(*)_1$]   $\varphi_{\alpha, \beta} \geq \varphi_{\alpha, \gamma}$, for all $\alpha < \beta < \gamma$ in $X$.
    \end{itemize}
       Again, applying clause (4)(a) to
 the sequence  $\{\varphi_{\alpha_{i}, \alpha_{ \kappa}}: i< \kappa\},$ we can find  some $i < j < \kappa$ such that  $\varphi_{\alpha_{i}, \alpha_{\kappa}} \geq \varphi_{\alpha_{j}, \alpha_{\kappa}}.$
 It follows that $\mathbf{c} (\alpha_{i},\alpha_{j},\alpha_{\kappa})= (1,1)$,  hence as
   $\mathbf{c} \rest [X]^{3}$ is constant, we have
   $\mathbf{c} (\alpha, \beta, \gamma)= (1,1)$, for all $\alpha < \beta < \gamma$  from $X$.
    In particular,
      \begin{itemize}
      	\item[$(*)_2$]   $\varphi_{\alpha, \gamma} \geq \varphi_{\beta,\gamma }$, for any $\alpha < \beta < \gamma$  in $X$.
      \end{itemize}
  Now, combining $(*)_1$ along with $(*)_2$, we get that for all  $\alpha < \beta < \gamma$  from $X$
   $$\varphi_{\alpha, \beta} \geq \varphi_{\alpha, \gamma} \geq  \varphi_{\beta,\gamma },$$
   and this completes the proof of (i).

   (ii): This is similar to case (i).
\end{proof}
\begin{remark}
By the Erdos-Rado partition theorem, see \cite{EHMR}, it suffices to take
$\lambda=\beth_2(\kappa)^+$.
\end{remark}

\section{$\kappa$-algebraic closure}\label{4}
\bigskip

In this section, and among other things, we define the concept of closure of a sequence inside an additive model
and present some properties of it.

\subsection{A closure operation with respect to  formulas}

\begin{definition}\label{n8}
	Suppose $\bff=(M, \mathscr{L},  \lambda, \kappa,  \theta, \Omega)$ is an additive frame, $\varp < \theta$ and     $\overline{a} \in {}^{\varp} M $. We define the closure of $\overline{a}$ in $M$ as the following:
	$$\rm{cl}(\overline{a}, M) := \{ b \in M: \varphi_{{\overline{a}^{\smallfrown} \langle b \rangle }}(M, \overline{a}) \ \text{has cardinality} \ < \kappa \}.$$
	\iffalse
	\begin{enumerate}[(a)]
		\item $\rm{cl}(\overline{a}, M) = \{ b \in M: \varphi_{{\overline{a}^{\smallfrown} \langle b \rangle }}(M, \overline{a}) \ \text{has cardinality} \ < \kappa \},$
		
		\item $\rm{afn}(b, \overline{a}) = \{ c \in M: \varphi_{\langle b \rangle^{\smallfrown}\overline{a}} = \varphi_{\langle c \rangle^{\smallfrown}\overline{a}} \},$
		
		\item $\rm{grp}(b, \overline{a}) = \{ c_{1} - c_{2}: \varphi_{\langle b \rangle^{\smallfrown }\overline{a}} = \varphi_{\langle c_{1} \rangle^{\smallfrown} \overline{a}} = \varphi_{\langle c_{2} \rangle^{\smallfrown} \overline{a}} \},$
		
		\item $\rm{Grp}(\overline{a}) =  \bigcup \{ \rm{grp}(b, \overline{a}): b \in \rm{cl}(\overline{a}) \},$
		
		\item  $\rm{GRp}(b, \overline{a}) = \bigcup \{ \rm{grp}(b, \overline{a}_{\bullet}): \varphi_{\overline{a}_{\bullet}} = \varphi_{\overline{a}} \ \text{and} \ \varphi_{b_{\bullet} \overline{a}_{\bullet}} = \varphi_{b \overline{a}} \}.$
	\end{enumerate}
	\fi
\end{definition}

In what follows we will use the following result several times:

\begin{lemma}\label{n11}
	Suppose $\bff=(M, \mathscr{L},  \lambda, \kappa,  \theta, \Omega)$ is a general frame as in Theorem \ref{L7}, $\varp < \theta$ and $\overline{a} \in {}^{\varp}M$. We assume in addition
	that $\kappa\in \Omega$ and $\cL$ is closed under $\exists^\kappa$.
	Then the  following three assertions are true:
	\begin{enumerate}[(a)]
		\item     the set $\rm{cl}(\overline{a}, M)$ has cardinality $< \kappa.$
		
		\item $\{ a_i: i < \varp    \} \subseteq \rm{cl}(\overline{a}, M).$
		
		\item  Assume $b \in \rm{cl}(\overa, M)$. Then  $\rm{cl}(\overa ^{\frown} \langle b \rangle, M) \subseteq \rm{cl}(\overa, M)$.
	\end{enumerate}
\end{lemma}

\begin{proof}
	(a).	Suppose
	not, and let
	$\overline{a} \in {}^{\varp}M$ be such that the set $\rm{cl}(\overline{a}, M)$ has cardinality $\geq \kappa.$ This gives us a family $\{b_\alpha:\alpha<\kappa\}\subseteq  \rm{cl}(\overline{a}, M)$. Define $\overline{a}_\alpha:=\overline{a}^{\smallfrown}\langle b_\alpha \rangle$.
In the light of Lemma \ref{k17},
	there is $\alpha_* < \kappa$ such that the set
	$$X:=\{ \beta < \kappa: \overline{a}_{\beta} \in \varphi_{\overline{a}_{\alpha_*}} (M) \}$$
	is unbounded in $\kappa.$
	So,
	$\{b_{\beta}:\beta\in X\}\subseteq\varphi_{\overline{a}_{\alpha_*}} (M,\overline{a})$, which implies that
	$ |\varphi_{\overline{a}_{\alpha_*}}(M,\overline{a})|\geq |X|=\kappa.$
	This contradicts the fact that $b_{\alpha_*} \in \rm{cl}(\overline{a}, M).$

	(b).  Let $i< \varp$, and define the formula $\psi(\overline{x}, y)$ as
$
	\psi(\overline{x}, y):= (y=x_i).
$
	Then $\psi(\overline{a}, M)=\{a_i\}$.
	It is also clear that $\psi(\overline{x}, y)$ is minimal with respect to this property. This implies that
$\varphi_{\overline{a}^{\frown} \langle a_i\rangle}(\overline{a}, M)=
	\psi(\overline{a}, M).$
 In particular, $|\varphi_{\overline{a}^{\frown} \langle a_i\rangle}(\overline{a}, M)|=1 < \kappa$,
	and consequently $a_i \in \rm{cl}(\overline{a}, M).$
	
	(c). Suppose $d \in \rm{cl}(\overa ^{\frown} \langle b \rangle, M)$. Thanks to Definition \ref{n8}, $\varphi_{\overa ^{\frown} \langle b \rangle^{\smallfrown} \langle d \rangle }(\overline{a}, b, M)$
	has cardinality less that $\kappa.$
	 As  $b \in \rm{cl}(\overa, M)$, clearly
  \begin{itemize}
	\item[$(\ast)_1$:] the set $ B:=\varphi_{{\overline{a}^{\smallfrown} \langle b \rangle }}(\overline{a},M)\emph{ has cardinality } <\kappa.$
	\end{itemize}	
	For $b_1\in B$ let $A_{b_1}:=\varphi_{\overa ^{\frown} \langle b  \rangle^{\smallfrown} \langle d \rangle }( \overline{a},b_1,M)$.
	We now show that
	  \begin{itemize}
		\item[$(\ast)_2:$] if $b_1\in B$
		then $A_{b_1}\emph{ has cardinality } <\kappa.$
	\end{itemize}

Assume towards a contradiction that
  $|A_{b_1}| \geq \kappa $ for some $b_1 \in B$. Reformulating this, means that:$$M\models\varphi_{\overa ^{\frown} \langle b\rangle}[\overa,b_1] \wedge\exists^\kappa  z \ \varphi_{\overa ^{\frown} \langle b  \rangle^{\smallfrown} \langle d \rangle }[ \overline{a},b_1,z ].$$
Let $$\psi(\overline{x},{y}):=\varphi_{\overa ^{\frown} \langle b\rangle}(\overline{x}, {y}) \wedge\exists^\kappa  z \ \varphi_{\overa ^{\frown} \langle b  \rangle^{\smallfrown} \langle d \rangle }( \overline{x}, {y},z ),$$and recall that $\psi\in\cL$ and $ \overa ^{\frown} \langle b_1\rangle  \in\psi(M)$.
Note that	$ \overa ^{\frown} \langle b \rangle  \notin\psi(M)$, as $ |\varphi_{\overa ^{\frown} \langle b  \rangle^{\smallfrown} \langle d \rangle }( \overline{a},b,M)|<\kappa.$
Next we bring the following claim:
	  \begin{itemize}
	\item[$(\ast)_{2.1}$:] if $\overa ^{\frown} \langle b_1\rangle  \notin \varphi_{\overa ^{\frown} \langle b\rangle}(M).$
\end{itemize}

 To see this we argue by  the way of contradiction that $ \overa ^{\frown} \langle b_1\rangle  \in\varphi_{\overa ^{\frown} \langle b\rangle}(M)$. This implies following the minimality condition that $$\varphi_{\overa ^{\frown} \langle b\rangle}(M)\subseteq \varphi_{\overa ^{\frown} \langle b_1\rangle}(M)\subseteq\psi(M).$$Consequently, $\overa ^{\frown} \langle b\rangle\in\psi(M)$.  This contradiction completes the proof of   $(\ast)_{2.1}
	. $ But, we have $\langle b_1\rangle \in\varphi_{\overa ^{\frown} \langle b\rangle}(\overa ,M)$. This yields that
		  \begin{itemize}
		\item[$(\ast)_{2.2}$:]   $\overa ^{\frown} \langle b_1\rangle  \in \varphi_{\overa ^{\frown} \langle b\rangle}(M).$
	\end{itemize}

	 But  $ (\ast)_{2.1} $ and $ (\ast)_{2.2} $ together lead to a contradiction. In sum, the desired property $(\ast)_2$ is valid.
\iffalse
By our assumptions,
\[
M \models \forall^\kappa \overline{x} \bigg( \varphi_{\overa ^{\frown} \langle b \rangle^{\smallfrown} \langle d \rangle }(\overline{a}, b, x_i)                      \rightarrow \exists i \neq j, x_i = x_j\bigg),
\]
and
\[
M \models \forall^\kappa \overline{y} \bigg( \varphi_{\overa ^{\frown} \langle b \rangle}(\overline{a}, y_i)                      \rightarrow \exists i \neq j, y_i = y_j\bigg),
\]

	This means that $M \models \exists^\kappa z \varphi_{\overa ^{\frown} \langle b  \rangle^{\smallfrown} \langle d \rangle }( \overline{a},b_1, z)$.
\fi	
Recalling that $\kappa$ is regular,
	it follows that
			  \begin{itemize}
		\item[$(\ast)_3$:]   $\bigcup_{b_1\in B} \varphi_{\overa ^{\frown} \langle b  \rangle^{\smallfrown} \langle d \rangle }( \overline{a},b_1,M)\emph{ has cardinality } <\kappa.$
	\end{itemize}
	We will show that:
				  \begin{itemize}
		\item[$(\dagger)$:]   $\varphi_{{\overline{a}^{\smallfrown} \langle d \rangle }}(\overline{a},M)\subseteq \bigcup_{b_1\in B} \varphi_{\overa ^{\frown} \langle b  \rangle^{\smallfrown} \langle d \rangle }( \overline{a},b_1,M),$
	\end{itemize}
	from which it will follow that
	$\varphi_{{\overline{a}^{\smallfrown} \langle d \rangle }}(\overline{a}, M)$
	has cardinality less that $\kappa,$ and hence by definition, $d \in \rm{cl}(\overa  , M)$. Let us prove $(\dagger)$. To this end,
	let $d_1\in\varphi_{{\overline{a}^{\smallfrown} \langle d \rangle }}(\overline{a}, M)$. This implies that $\overline{a}^{\frown}\langle d_1 \rangle\in\varphi_{{\overline{a}^{\smallfrown} \langle d \rangle }}(M)$.
	Clearly, $$M\models \exists y\varphi_{\overa ^{\frown} \langle b  \rangle^{\smallfrown} \langle d \rangle }[ \overline{a},y,d],$$
	hence,
	$$M\models \exists y\varphi_{\overa ^{\frown} \langle b  \rangle^{\smallfrown} \langle d \rangle }[ \overline{a},y,d_1].$$
This gives us some $b_1$ so that
	$$M\models \varphi_{\overa ^{\frown} \langle b  \rangle^{\smallfrown} \langle d \rangle }[ \overline{a},b_1,d_1].$$ Then ${b_1\in B}$,
and $d_1 \in \varphi_{\overa ^{\frown} \langle b  \rangle^{\smallfrown} \langle d \rangle }( \overline{a},b_1,M),$
 and consequently, $(\dagger)$ holds. We are done.
\end{proof}

\subsection{An algebraic structure  over the closure-operation}\footnote{The result of this subsection is independent from the rest of
	the paper.}

\begin{definition}\label{n81}
Suppose $\bff=(M, \mathscr{L},  \lambda, \kappa,  \theta, \Omega)$ is a general frame. 	For $\overline{a} \in {}^{<\theta}M $, we introduce the following:
	
	\begin{enumerate}
		
		\item  $\rm{afn}(b, \overline{a}) = \{ c \in \rm{cl}(\overline{a}, M): \varphi_{\overa^{\smallfrown}\langle b \rangle}   \geq  \varphi_{\overa^{\smallfrown} \langle c \rangle} \}.$
		
		\item Suppose $\bff$ is abelian. Then $\rm{grp}(b, \overline{a}) = \{ c_{1} - c_{2}: c_1, c_2 \in \rm{afn}(b, \overline{a}) \}.$
		
	%	\item $\rm{Grp}(\overline{a}) =  \bigcup \{ \rm{grp}(b, \overline{a}): b \in \rm{cl}(\overline{a}) \},$
		
	%	\item  $\rm{GRp}(b, \overline{a}) = \bigcup \{ \rm{grp}(b, \overline{c}): \varphi_{\overline{c}} = \varphi_{\overline{a}}  \text{~and~}  \varphi_{\overline{c}^{\smallfrown} \langle b \rangle} = \varphi_{\overline{a}^{\smallfrown} \langle b \rangle} \}.$
	\end{enumerate}
\end{definition}

\begin{hypothesis}In what follows, and up to the end of this section, let us assume that  $\bff=(M, \mathscr{L},  \lambda, \kappa,  \theta, \Omega)$ is an additive frame,
$\kappa \in \Omega$ and $\mathscr{L}$ is closed under $\exists^\kappa x$.
\end{hypothesis}

\begin{lemma}\label{n14}
Let $\overline{a} \in {}^{<\theta} M $ and $b\in \rm{cl}(\overline{a}, M)$. Then $\rm{afn}(b, \overline{a}) $ is a subset of $\rm{cl}(\overline{a}, M)$ and it is affine.
\end{lemma}

\begin{proof}
First, recall that 	``$\rm{afn}(b, \overline{a}) \subseteq \rm{cl}(\overline{a}, M)$'' holds by the definition. For the second phrase, we have
 to show that $\rm{afn}(b, \overline{a}) $ is closed under $x - y + z$. To see this, let $c_{i}  \in \rm{afn}(b, \overline{a})$ for $i = 1, 2, 3$ and  set $c := c_{1} - c_{2} + c_{3}.$ Since the frame
 is additive, 	 $\varphi_{\overline{a}^{\smallfrown}\langle b \rangle}(M)$ is a subgroup, see Definition \ref{k2}(7),
 so  $\varphi_{\overline{a}^{\smallfrown}\langle b \rangle}(M)$  is affine-closed.
 Consequently,  $$\overline{a}^{\smallfrown}\langle c \rangle = \overline{a}^{\smallfrown} \langle c_{1} \rangle - \overline{a}^{\smallfrown} \langle c_{2} \rangle + \overline{a}^{\smallfrown} \langle c_{3} \rangle \in \varphi_{\overline{a}^{\smallfrown}\langle b \rangle}(M).$$
According to the  minimality, $\varphi_{\overa^{\smallfrown}\langle b \rangle} \geq \varphi_{\overa^{\smallfrown} \langle c \rangle}$. Thanks to Lemma \ref{n11}, $c \in \rm{cl}(\overline{a}, M)$. Hence
$c\in\rm{afn}(b, \overline{a})$,
and we are done.
\end{proof}

% page 4.2

\begin{lemma}\label{n17}
The following holds:
 	\begin{enumerate}
 	
 	\item $\rm{grp}(b, \overa)$ is a subgroup of $M.$

\item $\rm{afn}(b, \overline{a}) = \{  b + d: d \in \rm{grp}(b, \overline{a}) \}.$	\end{enumerate}
\end{lemma}

\begin{proof}
For clause (1), let $c_{i}  \in \rm{grp}(b, \overline{a}), $ where $i = 1, 2$. Following its definition, there are  some $b_{i, 1}, b_{i, 2} \in \rm{afn}(b, \overline{a})$ such that $ c_{i} = b_{i, 1} - b_{i, 2}$. So,

   $$\begin{array}{ll}
c_{1} - c_{2} &= (b_{1, 1} - b_{1, 2}) - (b_{2, 1} - b_{2, 2})\\
&=  (b_{1, 1}- b_{1, 2} + b_{2, 2}) - b_{2, 1}.
\end{array}$$
According to Lemma \ref{n14}, we know $b^*_{2, 1}=b_{1, 1}- b_{1, 2} + b_{2, 2} \in \rm{afn}(b, \overline{a})$,
and hence by definition of $\rm{grp}(b, \overa)$,
\[
c_{1} - c_{2}= b^*_{2, 1} - b_{2, 1} \in \rm{grp}(b, \overa).
\]
%We apply this along with definition of $\rm{grp}(b, \overline{a}),$ and deduce that $c_{1} - c_{2} \in \rm{grp}(b, \overline{a})$, which suffice.

To prove clause (2), let $c \in \rm{afn}(b, \overline{a})$. As clearly $b \in \rm{afn}(b, \overline{a}),$ by clause (1), $-b + c \in \rm{grp}(b, \overa),$ hence
 $c = b + (-b+c) \in  \big\{  b + d: d \in \rm{grp}(b, \overline{a}) \big\}.$
Conversely, suppose  $d \in \rm{grp}(b, \overline{a})$. Due to its definition,
there are for some $c_1, c_2 \in \rm{afn}(b, \overline{a})$ so that $d= c_1 - c_2$. Consequently, in view of  Lemma \ref{n14}, we see
$$b+d = b -c_2+ c_1 \in \rm{afn}(b, \overline{a}).$$
The equality follows.
\end{proof}

% page 4.3
\iffalse
\begin{question}\label{n20}
Adopt the notation from Definition \ref{k2}(c). We ask the following three natural questions:
	\begin{enumerate}
	\item Suppose $\psi (y, \overline{x}_{[\varp]}) = \varphi_{\rm{grp}(b, \overline{a})}$. Is $\psi({}^{\varp}M, \overline{a}) = \rm{grp}(b, \overline{a})?$
	
	\item Is $\rm{cl}(\overline{a})$ affine?
\item Let $b \in \rm{cl}(\overline{a})$. Is $\varphi_{\rm{grp}(b, \overline{a})}(M)$ of cardinality $< \kappa?$
	\end{enumerate}
\end{question}
\fi

\begin{definition}\label{n5}
Let $\bff=(M, \mathscr{L},  \lambda, \kappa,  \theta, \Omega)$ be an additive frame such that
$\kappa \in \Omega$ and $\mathscr{L}$ is closed under $\exists^\kappa x$.
We say $\bff$ is very nice, if it satisfies the following extra properties:
\begin{enumerate}

	\item
$\cL_{\bff} = \cL_{\infty, \theta}^{\rm{pe}}(\tau_{M})$,  or just $\cL_{\bff}$ is closed under $\exists \overline{x}_u,$ up to equivalence, where $|u| < \theta$.
	
	\item For every $X \subseteq {}^{\varp}M$, there is a formula $\varphi_X(\overline{x})$, such that $X \subseteq \varphi_X(M)$,
and if $\psi(\overline{x})$ is such that $X \subseteq \psi(M),$ then $\varphi_X(M) \subseteq \psi(M)$.
	
	\end{enumerate}
\end{definition}

% page 4.6

\begin{discussion}
	By a repetition of the argument presented in  the proof of Theorem \ref{L7}
we know Definition \ref{n5}(2) holds, when $M$ is an $R$-module and $\bff$ is defined as in  Theorem \ref{L7}.
\end{discussion}

\begin{proposition}\label{n26}
  Suppose the additive frame $\bff=(M, \mathscr{L},  \lambda, \kappa,  \theta, \Omega)$ is very nice. The following conditions hold.
 	\begin{enumerate}
 \item Assume $b \in M$ and $\overa \in {}^{\varp}M$. Then there is some formula $\varphi(\overline{x}, y)$ with $\len(\overline{x})=\varp$ such that
 $\varphi(\overline{a}, M) = \rm{grp}(b, \overline{a}).$

 	\item If $c \in \rm{grp}(b, \overline{a})$ and $b \in \rm{cl}(\overline{a}, M)$, then $c \in \rm{cl}(\overline{a}, M).$ Moreover, $\rm{cl}(\overline{a}, M)$ is a subgroup of $M$.

   	\item Let $b \in \rm{cl}(\overline{a}, M)$. Then $\varphi_{\rm{grp}(b, \overline{a})}(M)$ is of cardinality $< \kappa.$

   	\end{enumerate}
\end{proposition}
\begin{proof}
(1): Define the formula $\varphi$ as
$$\varphi(\overline{x}, y) = (\exists {y}^{1}, y^{2})\bigg[ \varphi_{\overline{a}^{\smallfrown} \langle b \rangle}(\overline{x}, y^1) \wedge \varphi_{\overline{a} ^{\smallfrown} \langle b \rangle}(\overline{x}, y^2) \wedge y = y_{2} - y_{1}\bigg].$$
We show that $\varphi$ is as required. By Definition \ref{n5}, $\varphi(\overline{x}, y) \in \mathscr{L}.$
First, suppose that $b \in \rm{grp}(b, \overline{a})$, and let
$b_{1}, b_{2} \in \rm{afn}(b, \overline{a})$ be such that  $b = b_{2} - b_{1}$. Then $b_{1}, b_{2}$ witness $M \models \varphi[\overline{a}, b].$ Hence $$(\ast)_1\quad\quad\qquad\rm{grp}(b, \overline{a}) \subseteq \varphi(\overline{a}, M).$$
In order to prove the reverse inclusion, suppose that
$b \in \varphi(\overline{a}, M).$ This implies that $M \models \varphi[\overa, b]$. Take
$b_{1}, b_{2}$ be witness it, i.e., $b = b_2 - b_1$ and
  $$M \models \varphi_{\overline{a}^{\smallfrown} \langle b \rangle}[\overline{a}, b_1] \wedge \varphi_{\overline{a} ^{\smallfrown} \langle b \rangle}[\overline{a}, b_2].$$
On the other hand, for $l=1, 2$ we have
\[
M \models  \varphi_{\overline{a}^{\smallfrown} \langle b \rangle}[\overline{a}, b_l] \Rightarrow  \varphi_{\overline{a}^{\smallfrown} \langle b \rangle} \geq  \varphi_{\overline{a}^{\smallfrown} \langle b_l \rangle},
\]
 hence
$b_{l} \in \rm{afn}(b, \overline{a})$. Consequently, $b=b_2 -b_1 \in \rm{grp}(b, \overline{a}).$
As $b$ is arbitrary, we conclude that
$$(\ast)_2 \quad\qquad\quad \varphi(\overline{a}, M) \subseteq \rm{afn}(b, \overline{a}).$$
By $(\ast)_1$ and $(\ast)_2$, we have  $\varphi(M, \overline{a}) = \rm{afn}(b, \overline{a})$, and we are done.

(2): In the light of  Lemma \ref{n14}, it suffices to show that $\rm{cl}(\overline{a}, M)$ is a subgroup of $M$. To this end, let
$b_{1}, b_{2} \in \rm{cl}(\overline{a}, M),  \, \varp = \rm{lg}(\overline{a})$ and
let $\varphi(\overline{x}, y)$ be as in clause (1).
It is easily seen that:
    \begin{enumerate}
        \item[$\bullet_{1}$] $M \models \varphi[\overline{a}, b_2 - b_1],$

        \item[$\bullet_{2}$] $\varphi(\overline{a}, M)$ has cardinality $< \kappa.$
        \end{enumerate}
Thanks to the minimality condition,
$\varphi_{\overa ^{\smallfrown} \langle b_2 -b_1  \rangle}(M) \subseteq \varphi(\overline{a}, M).
$
In other words, $|\varphi_{\overa ^{\smallfrown} \langle b_2 -b_1  \rangle}(M)| < \kappa$, which implies that
     $b_{2} - b_{1} \in \rm{cl}(\overline{a}, M).$ We have proved
     $$b_{1}, b_{2} \in \rm{cl}(\overline{a}, M) \Rightarrow b_{2} - b_{1} \in \rm{cl}(\overline{a}, M).$$ Therefore,  $\rm{cl}(\rm{\overline{a}}, M)$ is a subgroup of $M.$

    (3): The proof is similar to the proof of clause (2).
\end{proof}

\section{On the absolutely co-Hopfian property}

This section is devoted to the proof of Theorem \ref{1.2} from the introduction.
We start by  recalling the   definition of (co)-Hopfian modules.

\begin{definition}
Let $M$ be an $R$-module.
	\begin{itemize}
		\item[(i)]
		$M$ is called \emph{Hopfian} if its surjective $R$-endomorphisms are automorphisms.
		\item[(ii)] $M$  is called
		\emph{co-Hopfian} if its injective $R$-endomorphisms are automorphisms.
	\end{itemize}
\end{definition}

This can be extended to:
\begin{definition}
	Let $M$ be a $\tau$-model.
	\begin{itemize}
		\item[(i)]
		$M$ is called \emph{Hopfian} if its surjective  $\tau$-morphisms are $\tau$-automorphisms.
		\item[(ii)] $M$  is called
		\emph{co-Hopfian} if its injective $\tau$-morphisms are $\tau$-automorphisms.
	\end{itemize}
\end{definition}

For the convenience of the reader, we present
the definition of potentially isomorphic, and
discuss some basic facts about them which are used in the paper, and only sketch  the proofs in most instances.
\begin{definition}\label{al}
Let  $M, N$ be two structures of our vocabulary.
Recall that  $M$ and $ N$  are called \it{potentially isomorphic} provided they are isomorphic in some forcing extension.\end{definition}

Recall that a group	$G$ is called   absolutely co-Hopfian (resp. Hopfian) if it is co-Hopfian  (resp. Hopfian) in any
further	generic extension of the universe.
\begin{discussion}\label{nad}
Suppose $M$ and $ N$  are  potentially isomorphic.
According to   \cite{marker} and \cite{Nad}
this is holds iff
 $[(M \models \varphi) \Longleftrightarrow  (N \models \varphi)] $ for every
 $\cL_{\infty, \aleph_0}$-sentence $\varphi.$  We denote this property by
$M \equiv_{\cL_{\infty, \aleph_0}} N.$
\end{discussion}

  The following is a simple variant of Discussion \ref{nad}. We state it in our context.
\begin{lemma}
\label{lem2}
Assume $M$ and $N$ are two $\tau$-structures.
\begin{enumerate}
\item Suppose for every sentence $\varphi \in \cL_{\infty, \aleph_0}(\tau)$, we have
$[(M \models \varphi)  \Longrightarrow (N \models \varphi)]$.
 Then there is an embedding of $M$ into $N$ in
$V[G_{\mathbb{P}}]$, where $\mathbb{P}$ collapses $|M|+|N|$ into $\aleph_0$.

\item In clause (1), it suffices to consider sentences $\varphi$ in the closure of base formulas under arbitrary conjunctions and $\exists x.$
\end{enumerate}
\end{lemma}
\begin{proof}
We give a proof for completeness. Let $M=\{a_n: n<\omega  \}$ be an enumeration of $M$ in $V[G_{\mathbb{P}}]$. By induction on $n$ we define a sequence
$\langle b_n: n<\omega \rangle$ of members of $N$ such that for each formula $\varphi(x_0, \cdots, x_{n})$ from $\cL_{\infty, \aleph_0}$,
\[
(*)_n \quad\quad\quad \big(M \models \varphi[a_0, \cdots, a_{n}]\big) \Rightarrow \big( N \models \varphi[b_0, \cdots, b_{n}]\big).
\]
Let
$
\Phi_0(x_0) = \bigwedge \{ \varphi(x_0): M \models \varphi[a_0]  \} \in \cL_{\infty, \aleph_0}.
$
Then $M \models \Phi_0(a_0)$, and by our assumption, there exists some $b_0 \in N$ such that $N \models \Phi_0(b_0)$.
Now suppose that $n<\omega$ and we have defined $b_0, \cdots, b_n.$ We are going to define $b_{n+1}$.
Let
\[
\Phi_{n+1}(x_0, \ldots, x_{n+1}) = \bigwedge \big\{ \varphi(x_0, \ldots, x_{n+1}): M \models \varphi[a_0, \cdots, a_{n+1}]  \big\}.
\]Clearly, $\Phi_{n+1}(x_0, \cdots, x_{n+1})  \in \cL_{\infty, \aleph_0} $.
Also, $$M \models \exists x_{n+1} \Phi_{n+1}(a_0, \cdots, a_n, x_{n+1}).$$ According to the induction hypothesis $(*)_n$,  we have
 $N \models \varphi[b_0, \cdots, b_{n+1}]$  for some $b_{n+1} \in N$.
This completes the construction of the sequence $\langle b_n: n<\omega \rangle$.  The assignment
$a_n\mapsto b_n$ defines a map $f: M \to N$ which is an embedding
of $M$ into $N$.
\end{proof}

\begin{fact}\label{ef}	
\begin{enumerate}  \item
Let $\lambda \geq \kappa_{\rm{beau}}$ and let $\langle G_{\alpha}: \alpha < \lambda \rangle$ be a sequence of $\tau$-models with $|\tau|< \kappa_{\rm{beau}}(\tau)$. Then   in some forcing extension $\bfV^{\bbP}, G_{\alpha}$ is embeddable into $G_{\beta},$  for some $\alpha < \beta < \lambda.$ Here, $\mathbb{P}$ collapses $| G_{\alpha}|+|G_{\beta}|$ into $\aleph_0$. Moreover, if $x_{\gamma} \in G_{\gamma}$ for $\gamma < \lambda$ then for some $\alpha < \beta  < \lambda,$ in some $\bfV^{\bbP}$ there is an embedding of $G_{\alpha}$ into $G_{\beta}$ mapping $x_{\alpha}$ to $x_{\beta}$\footnote{This explains why \cite{Fuc73} gets only indecomposable abelian groups (not endo-rigid).}.

\item	Suppose for $R$-modules $M$ and $N$ we have
 $[(M \models \varphi)  \Longrightarrow (N \models \varphi) ]$,  where   $\varphi \in \cL_{\infty, \theta}^{\rm{ce}}(\tau)$.
    Then there is an embedding of $M$ into $N$ in
$V[G_{\mathbb{P}}]$, where $\mathbb{P}$ collapses $|M|+|N|$ into $\aleph_0$.
\item Moreover, we can strengthen
the conclusion of part (2) to the following:
 	\begin{itemize}
 	\item[$(\ast)$] there is a $\mathbb{P}$-name $\pi$ satisfying:
 	\begin{itemize}
 		\item[$(\ast)_1$]
 		If $\overa\in(  {}^{<\theta }M)\cap V$
 		then $\pi(\overa)\in  ({}^{<\theta } N)\cap V$,
 		\item[$(\ast)_2$]$\Vdash_\mathbb{P}\pi$ maps
 		$\overa\in(  {}^{<\theta }M)\cap V$ onto  $\{\overb\in  {}^{<\theta }M:\rang(\overb) \subseteq \rang(\pi)\}.$
 	\end{itemize}
 \end{itemize}
\end{enumerate}
\end{fact}

\begin{proof}
For (1), see \cite{Sh:678}. Parts (2) and (3) are standard, see for example \cite{marker}.
\iffalse
(2). By Lemma \ref{lem2}, it suffices to show that for every sentence $\varphi \in \cL_{\infty, \aleph_0}$,
if $M \models \varphi,$ then $N \models \varphi.$
We argue by induction on the complexity of $\varphi$ that
\begin{center}
$(*)_\varphi$ \quad\quad\quad if $\phi \in \{\varphi, \neg\varphi\}$ and $M \models \phi$, then $N \models \phi.$
\end{center}
If $\varphi$ is an atomic formula, then it has the form
\[
\varphi:= \sum_{i \in I} r_ix_i=0,
\]
for some finite index set $I$ and some $r_i \in R$. Now, if $\phi=\varphi,$  this follows from our assumption that $\varphi \in \cL_{\infty, \theta}^{\rm{ce}}(\tau).$
Next, we deal with the formula $\phi=\neg\varphi.$ On the one hand, $\phi$ can be  written as
\[
\phi:= \exists y \big(y= \sum_{i \in I} r_ix_i \wedge y \neq 0 \big).
\]
On the other hand we observe $$\exists y \big(y= \sum_{i \in I} r_ix_i \wedge y \neq 0 \big) \in \cL_{\infty, \theta}^{\rm{ce}}(\tau).$$ Again, the claim is clear.

Suppose $(*)_\varphi$ holds.
It is routine
 to check that

	\begin{itemize}
	\item[$\bullet$]   $(*)_{\neg\varphi}$ holds,

	\item[$\bullet$]  $(*)_{\exists\varphi}$ holds.

\end{itemize}
Suppose  $(*)_{\varphi_i}$ holds for all $i\in I$. It remains to note that  $(*)_{\bigwedge_{i \in I}\varphi}$ holds as well.
So, we are done.

(3): this is similar to part (2).
\fi
	\end{proof}
\iffalse
\begin{discussion}
We may look at model theory essentially replacing ``isomorphic'' by ``almost isomorphic'', that is isomorphisms by potential isomorphisms, see Definition \ref{al}.
In \cite{Sh:12} we have suggested to reconsider (in this direction) a major theme in model theory, that of counting the number of isomorphism types.
\end{discussion}
\fi
Now, we are ready to prove:
\begin{theorem}\label{r2}The following assertions are valid:
	\begin{enumerate}  \item    If $M$ is an abelian group of cardinality $\geq \kappa := \kappa_{\rm{beau}}$, then $M$ is not absolutely co-Hopfian, indeed, after collapsing the size of $M$ into $\omega$, there is a one-to-one endomorphism $\varphi \in \rm{End}(M)$ which is not onto.

\item  If $M$ is an $R$-module of cardinality $\geq \kappa=\kappa_{\rm{beau}}(R)$,
then $M$ is not absolutely co-Hopfian.

\item	If $M$ is an $\tau$-model  of cardinality $\geq \kappa=\kappa_{\rm{beau}}(\tau)$,
then $M$ is not absolutely co-Hopfian.
\end{enumerate}
\end{theorem}

\iffalse
%\begin{notation}
%$\kappa_{\rm{beau}}(\theta) := \min \{ \kappa:  \kappa \geq \theta \ \text{is beautiful} \}.$
%\end{notation}

\begin{notation}
By 	$ A \| B$ we mean $A\subseteq B$ or $B\subseteq A$.
\end{notation}

\begin{theorem}\label{r2}
(1) and (2):    Assume $M \in \rm{AB}, \kappa = \kappa_{\rm{beau}} > \theta$ \underline{or} $M$ an additive $(< \theta)$-model, $\kappa = \kappa_{\rm{beau}}(\theta).$ Let $\Omega:=\{1,\kappa\}$. If $M$ has cardinality $\geq \kappa$ \underline{then} after some forcing there is a one-to-one endomorphism $\varphi \in \rm{End}(M)$ which is not onto.
\end{theorem}
\fi
\begin{proof}
Let $M$ be an abelian group or an $R$-module  of size $|M| \geq \kappa$.
 Thanks to Theorem  \ref{L7}, there exists  an additive  frame $\bff$ as there such that $M:=M_{\bff}$,
 $\cL_{\bff}=\cL_{\infty, \theta}^{\rm{co}}(\tau), \kappa_{\bff}=\kappa$,
 $\lambda_{\bff}=\beth_2(\kappa)^+$
 and  $\theta_{\bff}=\aleph_0.$

  The proof splits into  two cases:

 {\bf Case 1: for some $\varp < \theta,$ and $\overa, \overb \in {}^{\varp}M$, $\varphi_{\overa}(M) \supsetneqq \varphi_{\overb}(M)$}.

Consider the $\tau$-models $(M,  \overb)$ and $(M, \overa)$. Let  $\varphi \in \cL_{\infty, \theta}^{\rm{ce}}(\tau)$ be a sentence, and suppose that
$(M,  \overa) \models \varphi.$ As $\varphi_{\overa}(M) \supsetneqq \varphi_{\overb}(M)$, it follows that $(M,  \overb) \models \varphi.$
Thus by Fact \ref{ef}(2), working in the generic extension by $\mathbb{P}=\text{Col}(\aleph_0, |M|)$,
 there exists a one-to-one endomorphism $\pi \in \rm{End}(M)$ such that $\pi(\overa) = \overb.$

There is nothing to prove if
$\pi$ is not onto.
So, without loss of generality we may and do assume that  $\pi$ is onto. Then $\pi, \pi^{-1} \in \rm{Aut}(M)$ and $\pi^{-1}(\overb) = \overa$.
    We claim that $\varphi_{\overa} \leq \varphi_{\overb}$.
    Due to the minimality condition for  $\varphi_{\overa}$, it is enough to show that $\overa\in \varphi_{\overb}(M)$.
    To this end, recall that
    $\overb\in \varphi_{\overb}(M)$.
    By definition,
    $M \models \varphi_{\overline{b}}[\overline{b}]$.
    In the light of Lemma \ref{a5}(2) we observe that
     $M \models \varphi_{\overline{b}}[\pi^{-1}(\overline{b})]=\varphi_{\overline{b}}[  \overline{a} ].$
\iffalse
    We should be careful: If  $\theta_{\bff}>\aleph_{0}$, this does not follow, as $$   {}^{<\theta }M\cap V\neq  {}^{<\theta }M\cap V[G_{\mathbb{P}}], $$ but,  as $\theta_{\bff}=\aleph_{0}$ this is absolute.
\fi
By definition, this means that $\overa\in \varphi_{\overb}(M)$, as requested. Consequently,
   $\varphi_{\overa }(M) \subseteq\varphi_{\overb }(M)$, which contradicts our assumption.

{\bf Case 2: not case $1$}.

Given two sets $A, B \subseteq M$, by  $ A \| B$ we mean that $A\subseteq B$ or $B\subseteq A$.
So in this case,  the following holds:
\begin{center}
    $(*)$: \quad\quad $\forall \varp < \kappa $ and $\forall \overline{a},\overline{b}\in {}^{\varp}M$ we have $\bigg(\varphi_{\overa}(M)\|\varphi_{\overb}(M) \Rightarrow \varphi_{\overa}(M)=\varphi_{\overb}(M)\bigg).$
\end{center}
       Set  $\Gamma = \{ \varphi_{\overa}: \overa \in {}^{<\theta}M \}$.  Now, we have the following easy claim.

\begin{claim}
\label{cla1}
    $\Gamma$  is a set of cardinality $< \kappa.$
\end{claim}
\begin{PROOF}{\ref{cla1}}
To see this, set  $\Gamma_{\varp} := \{ \varphi_{\overa}: \overa \in {}^{\varp}M \}$. Clearly,    $\Gamma=\bigcup_{\varp<\theta}\Gamma_{\varp}$, and since
$\theta< \kappa$ and $\kappa$ is regular, it suffices  to show that $|\Gamma_{\varp}|<\kappa$ for all $\varp<\theta$.
Suppose not and search for a contradiction. Take $\varp<\theta$ be such that $|\Gamma_{\varp}|\geq\kappa$. This enables us to find a sequence $\langle \overline{a}_\alpha: \alpha < \kappa \rangle$ in $^{\varp}M$
such that
\begin{itemize}
	\item[$\bullet_1$] $\forall{\alpha}<\kappa,~\varphi_{{\overa}_\alpha}\in\Gamma_{\varp}$

	\item[$\bullet_2$] $\forall \alpha\neq \beta$, $\varphi_{{\overa}_\alpha}(M)\neq \varphi_{{\overa}_\beta}(M)$.
\end{itemize}
We apply the property presented in Definition \ref{k2}(4)(a) to the family $\{\varphi_{{\overa}_\alpha}    \}_{\alpha < \kappa}$,  to find some $\alpha < \beta < \kappa$ such that  $\varphi_{{\overa}_\beta}(M) \supseteq \varphi_{{\overa}_\alpha}(M).$ By $(\ast)$, this implies that
$\varphi_{{\overa}_\beta}(M) = \varphi_{{\overa}_\alpha}(M),$ which contradicts  $\bullet_2$.
\end{PROOF}

Let $\chi:=|M|\geq\kappa$, and let $\mathbb{P}:=\rm{Col}(\aleph_{0},\chi)$.
Forcing with $\mathbb{P}$, collapses
$|M|$ into  $\aleph_0$, i.e., for any $\mathbb{P}$-generic filter $G_{\mathbb{P}}$
over $V$, we have
\begin{center}
$V[G_{\mathbb{P}}]\models$``$M$ is  countable''.
 \end{center}
 We are going to show that in $V[G_{\mathbb{P}}]$, there exists a 1-1 map $\pi:M\to M$ which is not surjective. To this end, we define  approximations to the existence of such $\pi$:

    \begin{enumerate}
        \item[$\boxplus$] Let $\AP$ be the set of all triples $(\overa, \overb, c)$ such that:

        \begin{enumerate}[(a)]
            \item $\overa, \overb \in {}^{\varp}M$ for some $\varp < \theta $ with $a_i\neq a_j$, $b_i\neq b_j$ if $i\neq j$.

            \item $\varphi_{\overa} (\overline{x}) \equiv \varphi_{\overb}(\overline{x})$ (in $M$).

            \item $c \in M$ is such that $ c \notin \rm{cl}(\overb, M), $ i.e., $\varphi_{\overb^{\smallfrown} \langle c \rangle}(M,\overb) \ \text{has cardinality} \ \geq \kappa$.
        \end{enumerate}
    \end{enumerate}
\begin{claim}
\label{cla2}
$\AP \neq \emptyset.$
\end{claim}
\begin{PROOF}{\ref{cla2}}
According to   Lemma \ref{n11}(a),  $\rm{cl}(\overline{a}, M)$ has cardinality $< \kappa.$ In particular, $|\rm{cl}(\emptyset, M)|<\kappa$, and hence as
 $|M| \geq \kappa,$ we can find some $c\in M\setminus\rm{cl}(\emptyset, M)$, and consequently, $(\langle\rangle\langle\rangle,c)\in\AP.$
 The claim follows.
 \end{PROOF}

Next, we bring the following claim, which plays the key role in our proof.
\begin{claim}\label{cla3}
Suppose   $(\overa, \overb, c)\in \AP$ and  $d_{1} \in M$ is such that $d_1 \neq a_i,$ for all $i$. Then there is  some $d_{2} \in M$ such that  $ \big(\overa^{\smallfrown} \langle d_{1} \rangle, \overb^{\smallfrown}\langle d_{2} \rangle, c\big) \in \AP.$
\end{claim}
\begin{PROOF}{\ref{cla3}}
 Recall that in $M$ we have $\varphi_{\overa} (\overline{x}) \equiv \varphi_{\overb}(\overline{x})$.
    First, we use this to find $d \in M$ such that $$(\dagger)\quad\quad\quad M \models \varphi_{\overa^{\smallfrown} \langle d_{1} \rangle}[\overb, d].$$
    Indeed, we look at the formula $$\psi(\overline{x}):=\exists y \big[\varphi_{\overa^{\smallfrown}\langle d_{1} \rangle}(\overline{x},y) \wedge \bigwedge_i y \neq x_i           \big].$$Since $\overa\in \psi(M)$, and due to the minimality of $\varphi_{\overa}$ with respect to this property, we should have  $\overb\in \varphi_{\overb} (M)  =\varphi_{\overa} (M) \subseteq \psi(M).$
    In other words,  $M \models \exists y\varphi_{\overa^{\smallfrown} \langle d_{1} \rangle}(\overb, y).$
    Hence, for some
    $d \in M$ we must have  $M \models \varphi_{\overa^{\smallfrown} \langle d_{1} \rangle}[\overb, d]$, and $d \neq b_i,$ for all $i$.
    So, $(\dagger)$ is proved. From this,
    $\overb^{\smallfrown}\langle d \rangle\in\varphi_{\overa^{\smallfrown} \langle d_{1} \rangle}(M)$. This implies, using the minimality condition on formulas,  that
     $\varphi_{\overb^{\smallfrown}\langle d \rangle}(M)  \subseteq \varphi_{\overa^{\smallfrown}\langle d_{1} \rangle}(M) $. Combining this along with $(*)$ yields that
      $\varphi_{\overa^{\smallfrown}\langle d_{1} \rangle}(M)= \varphi_{\overb^{\smallfrown}\langle d \rangle}(M).$
   First, we deal with the case $d \in \rm{cl}(\overb, M)$. Thanks to   Lemma \ref{n11}(c), and recalling that       $\Omega$ contains $\{1,{\kappa}\}$,  we know $\rm{cl}(\overb ^{\frown} \langle d \rangle, M)  \subseteq  \rm{cl}(\overb, M)$. Consequently, $c \notin \rm{cl}(\overb^{\smallfrown}\langle d \rangle, M)$. Following definition, one has  $\big(\overa^{\smallfrown} \langle d_{1} \rangle, \overb^{\smallfrown} \langle d \rangle, c\big) \in \AP,$  and we are done  by taking $d_2=d.$
  So, without loss of generality let us assume that  $d \notin \rm{cl}(\overb, M).$ In particular, $d \neq b_i$ for all $i$. Since $(\overa, \overb, c)\in \AP$, we have  $c \notin \rm{cl}(\overb, M).$ According to $\boxplus$(c), the set
       $$\bfI := \bigg\{ c' \in M: M \models \varphi_{\overb^{\smallfrown} \langle c \rangle}[\overb, c'] \bigg\}$$
        has cardinality $\geq \kappa.$ Therefore, there is $c' \in \bfI \setminus \rm{cl}(\overb^{\smallfrown}\langle d\rangle, M).$ By the minimality condition and since $\overb^{\smallfrown}\langle c'   \rangle \in \varphi_{\overb^{\smallfrown}\langle c \rangle}(M)$, we have
        $\varphi_{\overb^{\smallfrown}\langle c' \rangle}(M) \subseteq  \varphi_{\overb^{\smallfrown}\langle c \rangle}(M)$.
     Now, we use this along with $(*)$ and deduce that
         $\varphi_{\overb^{\smallfrown}\langle c \rangle}(M) = \varphi_{\overb^{\smallfrown}\langle c' \rangle}(M).$   Then, in the same vein as above, we can find some   $d'$ such that
         \begin{enumerate}
        	\item[$\bullet$]  $d' \notin \{c\}\cup \{ b_i: i   \}$,
        	
        	\item[$\bullet$] $\varphi_{\overb^{\smallfrown} \langle c, d' \rangle}(M) = \varphi_{\overb^{\smallfrown} \langle c', d \rangle}(M)$,
        		\item[$\bullet$]
         $\varphi_{\overb^{\smallfrown} \langle  d \rangle}(M) = \varphi_{\overb^{\smallfrown} \langle  d' \rangle}(M)$.
        \end{enumerate}
  In fact, to obtain  such  $d'$,
 one needs to repeat the same argument as
  above but with $$\psi(\overline{x}, z) = \exists y \big[\varphi_{\overb^{\smallfrown}\langle c, d \rangle}(\overline{x}, z, y) \wedge \varphi_{\overb^{\smallfrown}\langle  d \rangle}(\overline{x}, y) \wedge (y \neq\overline{x})\big]. $$

Since	$c' \notin \rm{cl}(\overb^{\smallfrown}\langle d\rangle, M), $ these yield  that $c \notin \rm{cl}(\overb^{\smallfrown}\langle d' \rangle, M)$.  In summary, by letting
 $d_2=d'$ we have $ \big(\overa^{\smallfrown} \langle d_{1} \rangle, \overb^{\smallfrown}\langle d_{2} \rangle, c\big) \in \AP,$ as claimed.
\end{PROOF}

In $V[G_{\mathbb{P}}]$, $M$ is countable. Let us list $M$ as $\{a_i:i<\omega\}$. We define $\pi:M\to M$ by evaluating $\pi$ at $a_i$. We do this by induction, in such a way that for some fixed $c \in M$ and all $n<\omega$, if $\pi(a_i)=b_i$,  then
\[
(\dagger\dagger)_n \quad\quad\quad \big(\langle a_i:i<n\rangle,   \langle b_i:i<n\rangle, c\big) \in \AP.
 \]
 Recall from Claim \ref{cla2} and its proof  that there is some $c$ in $M$ such that
$(\langle\rangle\langle\rangle,c)\in\AP$. Let us apply Claim \ref{cla3} to

\begin{itemize}
	\item[$\bullet$] ${\overa}:=\langle\rangle$
	
		\item[$\bullet$] ${\overb}:=\langle\rangle$
	\item[$\bullet$] $d_1:=a_0$.
\end{itemize}

This gives us an element $b_0\in M$ such that
$ \big(\langle a_0\rangle,   \langle b_0\rangle, c\big) \in \AP.$  Let $\pi(a_0)=b_0$.
Now, suppose inductively we have defined
$\pi(a_i)=b_i$ for all $i<n$ such that $(\dagger\dagger)_n $ is true.
Let us apply Claim \ref{cla3} to

\begin{itemize}
	\item[$\bullet$] ${\overa}:=\langle a_i:i<n\rangle$,
	
	\item[$\bullet$] ${\overb}:=\langle b_i:i<n\rangle$,
	\item[$\bullet$] $d_1:=a_n$.
\end{itemize}

This gives us an element $b_n \in M$ such that
 $ \big( {\overa}^{\smallfrown} \langle a_n \rangle,  {\overb}^{\smallfrown} \langle b_n \rangle, c \big) \in \AP.$        Let $\pi(a_n)=b_n$, and note that
$(\dagger\dagger)_{n+1}$ holds as well. This completes the inductive definition of $\pi$.

 The proof becomes complete if we can show the following three items are satisfied:
\begin{itemize}
	\item[$(\maltese)_1$] $\pi$ is a homomorphism.

	\item[$(\maltese)_2$] $\pi$ is 1-to-1.
	
		\item[$(\maltese)_3$] $\pi$ is not surjective.
\end{itemize}
Let us check these:
 \begin{itemize}
 	\item[$(\maltese)_1$]
 	Suppose $\varphi=\varphi_{\overa}$ is a first order formula, hence, without  loss of generality,  the length of $\overa$ is finite and we can enlarge $\overa$ to $\langle a_i:i<n\rangle$ for some $n<\omega$.
 	Recall that $b_i:=\pi(a_i)$.  By the construction $$ \big(\langle a_j:j<n\rangle,   \langle b_j:j<n\rangle, c\big) \in \AP.$$
 	Assume $M\models \varphi[\overa]$. We  show that $M\models \varphi[\overb]$.
  We have ${\overb}\in\varphi_{\overb}(M)$ and by our construction, $\varphi_{\overb}(M)=\varphi_{\overa}(M) \subseteq\varphi(M)$,
  hence ${\overb}\in\varphi(M)$, which means that $M\models \varphi[\overb]$.
  From this, it immediately follows that $\pi$ is a homomorphism.
 	
 \item[$(\maltese)_2$] Following the construction  given by Claim \ref{cla3},
 	we can always find $b_n$ so that $b_n\neq b_i$ for all $i<n$. So, $\pi$ is 1-to-1.
 	
 	\item[$(\maltese)_3$]  Suppose by the way of contradiction that $\pi$ is  surjective. In particular, we can find some $n<\omega$
 such that $c=\pi(a_n)=b_n$. In the light of Lemma \ref{n11}(b) we observe that
$
 c=b_n \in \rm{cl}\big( \langle  b_i: i < n+1         \rangle, M\big),
$
 which contradicts the following fact$$\big(\langle  a_i: i < n+1         \rangle, \langle  b_i: i < n+1         \rangle, c\big) \in \AP.$$
 \end{itemize}
 The proof is now complete.

 (3): The proof of this item is similar to the first part.
\end{proof}

It may be worth mentioning that the following aspect of Problem 1.1(ii) remains open:

\begin{question}
Is it is possible to construct an absolutely
co-Hopfian torsion-free abelian group of size $\lambda$ for $\lambda < \kappa_{\rm{beau}}$?
\end{question}

\section*{Acknowledgements}
	The authors sincerely thank the referee for their thorough review of the paper and for providing valuable comments.

\end{document}